\documentclass[11pt]{article}

\usepackage[authoryear,round,compress,sort]{natbib}
\usepackage{amsmath,amssymb,amsfonts, amsbsy, amsthm, epsfig, subfig, graphicx,rotating, tabularx, booktabs, titling}
\usepackage[usenames]{color}
\usepackage{enumerate}
\usepackage{float}

\RequirePackage[colorlinks,citecolor=blue,urlcolor=red]{hyperref}
\usepackage{authblk}
\usepackage{multirow}

\usepackage{fancyhdr}
\newcommand{\noi}{\noindent}

\newcommand{\la}{\lambda}

\newcommand{\beq}{\begin{eqnarray*}}
\newcommand{\eeq}{\end{eqnarray*}}
\newcommand{\beqn}{\begin{eqnarray}}
\newcommand{\eeqn}{\end{eqnarray}}

\newcommand{\var}{{\rm var}}
\newcommand{\bi}{\begin{itemize}}
\newcommand{\ei}{\end{itemize}}

\newcommand{\be}{\begin{equation}}
\newcommand{\ee}{\end{equation}}
\newcommand{\nn}{\nonumber}

\newcommand{\bG}{\mathbb{G}}

\newcommand{\bpsi} {\boldsymbol{\psi}}

\newcommand{\ignore}[1]{}{}

\newcommand{\sn}{\sum_{i=1}^n}
\newcommand{\sm}{\sum_{j=1}^m}

\newcommand{\de}{\delta}

\newcommand{\ga}{\gamma}

\newcommand{\s}{\sqrt}
\newcommand{\e}{\mathbb{E}}

\newcommand{\br}{\mathbb{R}}

\newcommand{\bxi} {\boldsymbol{\xi}}

\newcommand{\argmax}{\mathop{\rm arg\max}}

\numberwithin{equation}{section}
\theoremstyle{plain}

\newtheorem{lemma}{Lemma}[section]
\newtheorem{proposition}{Proposition}[section]
\newtheorem{Example}{Example}
\newtheorem{theorem}{Theorem}[section]
\newtheorem{assumption}{Assumption}[section]

\newtheorem{remark}{Remark}[section]

\begin{document}

\bibliographystyle{plainnat}

\title{Two-Sample Smooth Tests for the Equality of Distributions}

\author{Wen-Xin Zhou\thanks{Department of Operations Research and Financial Engineering, Princeton University, Princeton, NJ 08544, USA. E-mail: wenxinz@princeton.edu}~,
 Chao Zheng \thanks{School of Mathematics and Statistics, University of Melbourne, Parkville, VIC 3010, Australia. E-mail: zhengc1@student.unimelb.edu.au}~,
 and Zhen Zhang \thanks{Department of Statistics, University of Chicago, Chicago, IL 60637, USA. E-mail: zhangz19@galton.uchicago.edu}
}


\date{}

\maketitle

\vspace{-1.5cm}

\begin{abstract}
This paper considers the problem of testing the equality of two unspecified distributions. The classical omnibus tests such as the Kolmogorov-Smirnov and Cram\'er-von Mises are known to suffer from low power against essentially all but location-scale alternatives. We propose a new two-sample test that modifies the Neyman's smooth test and extend it to the multivariate case based on the idea of projection pursue.  The asymptotic null property of the test and its power against local alternatives are studied. The multiplier bootstrap method is  employed to compute the critical value of the multivariate test. We establish validity of the bootstrap approximation in the case where the dimension is allowed to grow with the sample size. Numerical studies show that the new testing procedures perform well even for small sample sizes and are powerful in detecting local features or high-frequency components.
\end{abstract}

\noindent
{\bf Keywords}: {\/ Neyman's smooth test; Goodness-of-fit; Multiplier bootstrap; High-frequency alternations; Two-sample problem}


\section{Introduction}
\label{intro.sec}

Let $X$ and $Y$ be two $\br^p$-valued random variables with continuous distribution functions $F $ and $G$, respectively, where $p\geq 1$ is a positive integer. Given data from each of the two {\it unspecified} distributions $F$ and $G$, we are interested in testing the null hypothesis of the equality of distributions

\be
	 H_0: F= G \ \  \mbox{ versus } \ \ H_{1}: F \neq G.   \label{test.dist}
\ee
This is the two-sample version of the conventional goodness-of-fit problem, which is one of the most fundamental hypothesis testing problems in statistics \citep{Lehmann_Romano_2005}.

\subsection{Univariate case: $p=1$}
\label{univariate.intro.sec}

Suppose we have two independent univariate random samples $\mathcal{X}_n =\{X_1,\ldots, X_n\}$ and $\mathcal{Y}_m=\{Y_1,\ldots, Y_m\}$ from $F$ and $G$, respectively. The empirical distribution functions (EDF) are given by
\be
	F_n(x) = \frac{1}{n} \sn I(X_i \leq x ) \ \ \mbox{ and }  \ \ G_m(y) = \frac{1}{m} \sm I(Y_j \leq y).  \label{edfs}
\ee
For testing the equality of two univariate distributions, conventional approaches in the literature use a measure of discrepancy between $F_n$ and $G_m$ as a test statistic. Prototypical examples include the Kolmogorov-Smirnov (KS) test $\hat \Psi_{{\rm KS}} = \s{nm/(n+m)} \sup_{t \in \br}  | F_n(t) - G_m(t) |$ and the Cram\'er-von Mises (CVM) family of statistics
$$
		\hat \Psi_{{\rm CVM}} = \frac{nm}{n+m} \int_{-\infty}^{\infty}  \{ F_n(t) -G_m(t)  \}^2 w(H_{n+m}(t)) \, d H_{n,m}(t),
$$
where  $H_{n,m}(t):=\{n F_n(t)+mG_m(t)\}/(n+m)$ denotes the pooled EDF and $w $ is a non-negative weight function. Taking $w \equiv 1$ yields the Cram\'er-von Mises statistic, and $w(t)=\{t(1-t)\}^{-1}$ yields the Anderson-Darling statistic \citep{darling1957kolmogorov}.

The traditional omnibus tests, which have been widely used for testing the two-sample goodness-of-fit hypothesis \eqref{test.dist} due to their simplicity with which they can be performed, suffer from low power in detecting densities containing high-frequency components or local features such as bumps, and thus may have poor finite sample power properties \citep{Fan_1996}. It is known from empirical studies that the CVM test has poor power against essentially all but location-scale alternatives \citep{Eubank_LaRiccia_1992}. The same issue arises in the KS test as well. To enhance power under local alternatives, Neyman's smooth method \citep{Neyman_1937}  was introduced earlier than the traditional omnibus tests, to test only the first $d$-dimensional sub-problem if there is prior that most of the discrepancies fall within the first $d$ orthogonal directions. Essentially, Neyman's smooth tests represent a compromise between omnibus and directional tests. As evidenced by numerous empirical studies over the years, smooth tests have been shown to be more powerful than traditional omnibus tests over a broad range of realistic alternatives. See, for example, \cite{Eubank_LaRiccia_1992}, \cite{Fan_1996}, \cite{Janssen_2000}, \cite{Bera_Ghosh_2002} and \cite{Escanciano_2009}.

A two-sample analogue of the Neyman's smooth test was recently proposed by \cite{Bera_Ghosh_Xiao_2013} for testing the equality of $F$ and $G$ based on two independent samples. The test statistic is asymptotically chi-square distributed and as a special case of Rao's score test, it enjoys certain optimality properties. Specifically, \cite{Bera_Ghosh_Xiao_2013} motivated the two-sample Neyman's smooth test by considering the random variable $V=F(Y)$ with distribution and density functions given by
\be
	H(z):=G(F^{-1}(z)) \ \ \mbox{ and } \ \  \rho(z):=g(F^{-1}(z))/f(F^{-1}(z))  \label{density.rho}
\ee
for $0<z<1$, respectively, where $F^{-1}$ is the quantile function of $X$, i.e. $F^{-1}(z)=\inf\{ x\in \mathbb{R}: F(x)\geq z\}$, and $f$ and $g$ denote the density functions of $X$ and $Y$. Assume that $F$ and $G$ are strictly increasing, then $H$ is also increasing, $\rho(z)\geq 0$ for $0<z<1$ and $\int_0^1 \rho(z)\, dz=1$. Under the null hypothesis $H_0$, $\rho \equiv 1$ so that $V =_d U(0,1)$. In other words, the null hypothesis $H_0$ in \eqref{test.dist} is equivalent to
\be
	\tilde{H}_0:  \rho(z) = 1 \ \ \mbox{ for all }  0< z <1,   \label{equivalent.null}
\ee
where $\rho$ is as in \eqref{density.rho}. Throughout, the function $\rho$ is referred as the ratio density function. Based on Neyman's smooth test principle, we restrict attention to the following smooth alternatives to the null of uniformity
\be
	\rho_{\theta}(z) = C_d(\theta) \exp\bigg\{ \sum_{k=1}^d \theta_k \psi_k(z) \bigg\} \ \ {\rm for } \ \  \theta:= (\theta_1,\ldots, \theta_d)^{\intercal} \in \br^d \ \  {\rm and } \ \ 0< z < 1, \label{alternatives}
\ee
which include a broad family of parametric alternatives, where $d=\mbox{dim}(\theta)$ is some positive integer and $\{ C_d(\theta) \}^{-1} = \int_0^1 \exp\big\{ \sum_{k=1}^d \theta_k \psi_k(z) \big\} \, dz$. Setting $\psi_0 \equiv 1$, the functions $\psi_1,\ldots, \psi_d$ are chosen in such a way that $\{\psi_0, \psi_1,\ldots, \psi_d \}$ forms a set of orthonormal functions, i.e.
\be
	\int_0^1 \psi_k(z) \psi_{\ell}(z) \, dz = \de_{k \ell} = \begin{cases}
 1, \ \  \mbox{ if }  k=\ell ,  \\
 0, \ \ \mbox{ if }  k \neq \ell.
\end{cases}  \label{orthonormality}
\ee
The null hypothesis asserts $H^d_{0}: \theta=0$. Assuming that $m\leq n$ and the truncation parameter $d$ is fixed, the two-sample smooth test proposed by \cite{Bera_Ghosh_Xiao_2013} is defined as $\hat \Psi_{{\rm BGX}} = m \hat{\bpsi}^{\intercal} \hat{\bpsi}$, where $\hat{\bpsi}= m^{-1} \sum_{j=1}^m \bpsi(\hat V_j )$, $\bpsi=(\psi_1,\ldots, \psi_d)^{\intercal}$ and $\hat V_j=F_n(Y_j)$. Under certain moment conditions and if the sample sizes $(n,m)$ satisfy $ m \log\log n =o(n)$ as $n, m\rightarrow \infty$, the test statistic $\hat \Psi_{{\rm BGX}}$ converges in distribution to the $\chi^2$ distribution with $d$ degrees of freedom. Accounting for the error of estimating $F$, \cite{Bera_Ghosh_Xiao_2013} further considered a generalized version of the smooth test that is asymptotically $\chi^2(d)$ distributed and can be applied when $n$ and $m$ are of the same magnitude.

However, \cite{Bera_Ghosh_Xiao_2013} only focused on the fix $d$ scenario (i.e. $d=4$) so that their two-sample smooth test is consistent in power against alternative where $V=F(Y)$ does not have the same first $k$ moments as that of the uniform distribution \citep{Lehmann_Romano_2005}. If there is a priori evidence that most of the energy is concentrated at low frequencies, i.e. large $\theta_k$ are located at small $k$, it is reasonable to use Neyman's smooth test. Otherwise, Neyman's test is less powerful when testing contiguous alternatives with local characters \citep{Fan_1996}. As \cite{Janssen_2000} pointed out, achieving reasonable power over more than a few orthogonal directions is hopeless. Indeed, the larger the value of $d$, the greater the number of orthogonal directions used to construct the test statistic. Therefore, it is possible to obtain consistency against all distributions if the truncation parameter $d$ is allowed to increase with the sample size. Motivated by \cite{chang2014simulation}, we regard $H^d_0:\theta=0$ as a global mean testing problem with dimension increasing with the sample size. When $d$ is large, Neyman's smooth test which is based on the $\ell_2$-norm of $\theta$ may also suffer from low powers under sparse alternatives. In part, this is because that the quadratic statistic accumulates high-dimensional estimation errors under $H^d_0$, resulting in large critical values that can dominate the signals under sparse alternatives.

To overcome the foregoing drawbacks, we first note that the traditional omnibus tests aim to capture the differences of two entire distributions as opposed to only assessing a particular aspect of the distributions, by contrast, Neyman's smooth principle reduces the original nonparametric problem to a $d$-dimensional parametric one. Lying in the middle, we are interested in enhancing the power in detecting two adjacent densities where one has local features or contains high-frequency components, while maintaining the same capability in detecting smooth alternative densities as the traditional tests. We expect to arrive at a compromise between desired significance level and statistical power by allowing the truncation parameter $d$ to increase with sample sizes. In Section~\ref{method.sec}, we introduce a new test statistic by taking maximum over $d$ univariate statistics. The limiting null distribution is derived under mild conditions, while $d$ is allowed to grow with $n$ and $m$. To conduct inference, a novel intermediate approximation to the null distribution is proposed to compute the critical value. In fact, when $n$ and $m$ are comparable, $d$ can be of order $n^{1/4}$ (resp. $n^{1/9}$) (up to logarithmic in $n$ factors) if the trigonometric series (resp. Legendre polynomial series) is used to construct the test statistic.

\subsection{Multivariate case: $p\geq 2$}
\label{multivariate.intro.sec}

As a canonical problem in multivariate analysis, testing the equality of two multivariate distributions based on the two samples has been extensively studied in the literature that can be dated back to \cite{Weiss_1960}, under the conventional fix $p$ setting. \cite{Friedman_Rafsky_1979} constructed a two-sample test based on the minimal spanning tree formed from the interpoint distances and their test statistic was shown to be asymptotically distribution free under the null; \cite{Schilling_1986} and \cite{Henze_1988} proposed nearest neighbor tests which are based on the number of times that the nearest netghbors come from the same group; \cite{Rosenbaum_2005} proposed an exact distribution-free test based on a matching of the observations into disjoint pairs to minimize the total distance within pairs. Work in the context of nonparametric tests include that of \cite{Hall_Tajvidi_2002}, \cite{Baringhaus_Franz_2004,Baringhaus_Franz_2010} and \cite{Biswas_Ghosh_2014}, among others.

Most aforementioned existing methods are tailored for the case where the dimension $p$ is fixed. Driven by a broad range of contemporary statistical applications, analysis of high-dimensional data is of significant current interest. In the high-dimensional setting, the classical testing procedures may have poor power performance, as evidenced by the numerical investigations in \cite{Biswas_Ghosh_2014}. Several tests for the equality of  two distributions in high dimensions have been proposed. See, for example, \cite{Hall_Tajvidi_2002} and \cite{Biswas_Ghosh_2014}. However, limiting null distributions of the test statistics introduced in \cite{Hall_Tajvidi_2002} and \cite{Biswas_Ghosh_2014} were derived when the dimension $p$ is fixed.

In the present paper, we propose a new test statistic that extends Neyman's smooth test principle to higher dimensions based on the idea of projection pursue. To conduct inference for the test, we employ the multiplier (wild) bootstrap method which is similar in spirit to that used in \cite{Hansen_1996} and \cite{Barrett_Donald_2003}. We refer to Section~\ref{multivariate.sec} for details on methodologies. It can be shown that (Propositions~\ref{limiting.null.supremum.statistic} and \ref{multiplier.bootstrap.theorem}), under mild conditions, the error in size of our multivariate smooth test decays polynomially in sample sizes $(n,m)$. It is noteworthy that we allow the dimension $p$ to grow as a function of $(n,m)$, a type of framework the existing methods do not rigorously address. More importantly, we do not limit the dependency structure among the coordinates in $X$ and $Y$ and no shape constraints of the distribution curves are known as a priori which inhibits a pure parametric approach to the problem.

\subsection{Organization of the paper}
\label{organization.sec}

The rest of the paper is organized as follows. In Section~\ref{method.sec}, we describe the two-sample smooth testing procedure in the univariate case. An extension to the multivariate setting based on projection pursue is introduced in Section~\ref{multivariate.sec}. Section~\ref{theory.section} establishes theoretical properties of the proposed smooth tests in both univariate and multivariate settings. Finite sample performance of the proposed tests is investigated in Section~\ref{numerical.sec} through Monte Carlo experiments. The proofs of the main results are given in Section~\ref{proof.section} and some additional technical arguments are contained in the Appendix.

\medskip
\noindent
{\sc Notation.} For a positive integer $p$, we write $[p]=\{1, 2, \ldots, p\}$ and denote by $|\cdot|_2$ and $|\cdot|_\infty$ the $\ell_2$- and $\ell_\infty$-norm in $\br^p$, respectively, i.e. $|x|_2=( x_1^2 +\cdots+ x_p^2)^{1/2}$ and $|x|_\infty=\max_{k\in [p]}|x_k|$ for $x=(x_1,\ldots, x_p)^{\intercal} \in \br^p$. The unit sphere in $\br^p$ is denoted by $\mathcal{S}^{p-1}=\{ x\in \br^p : |x|_2=1\}$. For two $\br^p$-valued random variables $X$ and $Y$, we write $X =_d Y $ if they have the same probability distribution and denote by $P_X$ the probability measure on $\br^p$ induced by $X$. For two real numbers $a$ and $b$, we use the notation $a\vee b=\max(a,b)$ and $a\wedge b=\min(a,b)$.  For two sequences of positive numbers $a_n$ and $b_n$, we write $a_n \asymp b_n$ if there exist constants $c_1, c_2>0$ such that for all sufficiently large $n$, $c_1\leq a_n/b_n \leq c_2$, we write $a_n =  O(b_n)$ if there is a constant $C>0$ such that for all $n$ large enough, $a_n \leq C b_n$, and we write $a_n \sim b_n$ or $a_n \simeq b_n$ and $a_n = o(b_n)$, respectively, if $\lim_{n \rightarrow \infty} a_n/b_n =1$ and $\lim_{n\rightarrow\ \infty} a_n/b_n=0$. For any two functions $f, g:\br \mapsto \br$, we denote with $f\circ g$ the composite function $f\circ g(x)=f\{g(x)\}$ for $x\in \br$.

For any probability measure $Q$ on a measurable space $(\mathbb{S}, \mathcal{S})$, let $\| \cdot \|_{Q,2}$ be $L_2(Q)$-seminorm defined by $\| f\|_{Q,2}= (Q |f|^2)^{1/2} = (\int |f|^2 \, dQ)^{1/2}$ for $f\in L_2(Q)$. For a class of measurable functions $\mathcal{F}$ equipped with an envelope $F(s) = \sup_{f\in \mathcal{F}} |f(s)|$ for $s\in \mathbb{S}$, let $N(\mathcal{F}, L_2(Q), \varepsilon \| F\|_{Q,2} )$ denote the $\varepsilon$-covering number of the class of functions $\mathcal{F}$ with respect to the $L_2(Q)$-distance for $0<\varepsilon\leq 1$. We say that the class $\mathcal{F}$ is {\it Euclidean} or {\it VC-type} \citep{Nolan_Pollard_1987,van der Vaart_Wellner_1996} if there are constants $A, v>0$ such that $\sup_Q N(\mathcal{F}, L_2(Q), \varepsilon \| F\|_{Q,2} )\leq (A/\varepsilon)^v$ for all $0<\varepsilon\leq 1$,  where the supremum ranges over  over all finitely discrete probability measures on $(\mathbb{S},\mathcal{S})$. When $\mathbb{S}=\br^p$ for $p\geq 1$, we use $\mathcal{S}$ to denote the Borel $\sigma$-algebra unless otherwise stated.

\section{Testing equality of two univariate distributions}
\label{method.sec}

\subsection{Oracle procedure}
\label{oracle.sec}

Without loss of generality, we assume $n \geq m$ and recall that the null hypothesis $H_0:F=G$ is equivalent to $\tilde H_0: V =_d U(0,1)$ for $V=F(Y)$ as in \eqref{equivalent.null}. Following \cite{Bera_Ghosh_Xiao_2013}, we consider the smooth alternatives lying in the family of densities \eqref{alternatives} which is a $d$-parameter exponential family, where $d=d_{n,m}$ is allowed to increase with $n$ and $m$ in order to obtain power against a large array of alternatives. In particular, this family is quadratic mean differentiable at $\theta=0$ and therefore the score vector at $\theta=0$ is given by $m^{-1/2} \big(  \frac{\partial}{\partial \theta_1} \log L_m(\theta) , \ldots,  \frac{\partial}{\partial \theta_d} \log L_m(\theta) \big) |_{\theta=0}$ \citep{Lehmann_Romano_2005}, where $L_m(\theta)=\{ C_d(\theta) \}^m \exp  \big\{ \sm \sum_{k=1}^d \theta_k \psi_k(V_j) \big\}$ is the likelihood function and $V_j=F(Y_j)$, such that $ \frac{\partial}{\partial \theta_k}  \log L_m(\theta) = \sm [ \psi_k(V_j) - E_{\theta} \{  \psi_k(V_j) \} ]$. As $\{\psi_0 \equiv 1, \psi_1,\ldots, \psi_d \}$ forms a set of orthonormal functions, it is easy to see that if $\theta=0$, $E_{\theta} \{  \psi_k(V) \} =0$ and $E_{\theta} \{  \psi_k(V)^2 \} =1$ for every $k\in [d]$.

To provide a more omnibus test against a broader range of alternatives, we allow a large truncation parameter $d$ and for the reduced null hypothesis $H^d_{0}: \theta =0$, it is instructive to consider the following oracle statistic
\be
	  \Psi(d)  = \max_{1\leq k\leq d }  \bigg| \frac{1}{\sqrt{m}}\sm \psi_k(V_j) \bigg|, \label{oracle.test}
\ee
which can be regarded as a {\it smoothed} version of the KS statistic. Throughout, the number of orthogonal directions $d=\mbox{dim}(\theta)$ is chosen such that $d\leq n \wedge m$. Intuitively, this extreme value statistic is appealing when most of the energy (non-zero $\theta_k$) is concentrated on a few dimensions but with unknown locations, meaning that both low- and high-frequency alternations are possible. Now it is a common belief \citep{tony2014two} that maximum-type statistics are powerful against sparse alternatives, which in the current context is the case where the two densities only differ in a small number of orthogonal directions (not necessarily in the first few). To see this, consider a contiguous alternative where there exists some $\ell^*$ such that $\theta_{\ell^*} \neq 0$ with $|\theta_{\ell^*}|$ sufficiently small and $\theta_\ell=0$ for all other $\ell$, then informally we have $C_d^{-1}(\theta) = \int_0^1 \exp \{ \theta_{\ell^*} \psi_{\ell^*}(z)  \} \, dz \simeq  \int_0^1  \{ 1+  \theta_{\ell^*} \psi_{\ell^*}(z)  \} \, dz =1$ and
\begin{align*}
	E_{\theta} \{  \psi_k(V) \}  = C_d(\theta) \int_0^1 \psi_k(z) \exp \{ \theta_{\ell^*} \psi_{\ell^*}(z)  \} \, dz  \simeq  \int_0^1 \psi_{k}(z)  \{ 1+  \theta_{\ell^*} \psi_{\ell^*}(z)  \} \, dz =  \theta_{\ell^*} \, \de_{k \ell^*}  .
\end{align*}
Under a sparse alternative where only a few components of $\theta=(\theta_1,\ldots, \theta_d)^{\intercal}$ are non-zero, the power typically depends on the magnitudes of the signals (non-zero coordinates of $\theta$) and the number of the signals.

\subsection{Data-driven procedure}
\label{data-driven.sec}

For the oracle statistic $\Psi(d)$ in \eqref{oracle.test}, the random variables $V_j=F(Y_j)$ are not directly observed as the distribution function $F$ is {\it unspecified}. Indeed, this is the major difference of the two-sample problem from the classical (one-sample) goodness-of-fit problem. We therefore consider the following data-driven procedure. In the first stage, an estimate $\hat{V}_j$ of $V_j$ is obtained by using the empirical distribution function $F_n$:
\be
	\hat V_j = F_n(Y_j) = \frac{1}{n} \sn I(X_i\leq Y_j).   \label{generated.Z}
\ee
Then the data-driven version of $\Psi(d)$ in \eqref{oracle.test} is defined by
\be
	\hat \Psi  =\hat{\Psi}(d)= \s{\frac{nm}{n+m}}\max_{ 1\leq k\leq d } | \hat \psi_k  | . \label{data-driven.test}
\ee
where $\hat{\psi}_k=m^{-1}\sm \psi_k(\hat V_j) $. In the case $m> n$, we may use $G_m $ instead of $F_n $, leading to an alternative test statistic $ \tilde{\Psi}(d)= \s{nm/(n+m)}\max_{ 1\leq k\leq d }  | n^{-1}\sn \psi_k(G_m(X_i) ) |$.

Typically, large values of $\hat \Psi$ lead to a rejection of the null $H^d_{0}:\theta=0$ and hence of $H_0:F=G$, or equivalently, $\tilde{H}_0$ in \eqref{equivalent.null}. For conducting inference, we need to compute the critical value so that the corresponding test has approximately size $\alpha$. A natural approach is to derive the limiting distribution of the test statistic $\hat{\Psi}(d)$ under the null. Under certain smoothness conditions on $\psi_k$, it can be shown that for every $k\in [d]$,
\begin{align*}
	\hat{\psi}_k = \frac{1}{m} \sm \psi_k(\hat{V}_j) \simeq \frac{1}{m}\sm \psi_k(V_j) - \frac{1}{n} \sn \psi_k(U_i) ,
\end{align*}
where $U_i = G(X_i)$. See, for example, \eqref{dec.1} and \eqref{Hoeffding.dec} in the proof of Proposition~\ref{thm.limiting.dist}. Under $H_0$ and when $d\geq 1$ is fixed, a direct application of the multivariate central limit theorem is that as $n,m\rightarrow \infty$,
\be
	\sqrt{\frac{nm}{n+m}} \big( \hat{\psi}_1, \ldots, \hat{\psi}_d \big)^{\intercal}    \xrightarrow d  \mathbf{G} =_d  N\big( 0,\mathbf{I}_d \big). \label{multivariate.CLT}
\ee
where $\mathbf{I}_d$ is the $d$-dimensional identity matrix. This implies by the continuous mapping theorem that $\hat{\Psi}(d) \xrightarrow d |\mathbf{G}|_\infty$ when $d$ is fixed.

For every $0<\alpha<1$, denote by $z_\alpha$ the $(1-\alpha)$-quantile of the standard normal distribution, i.e. $z_\alpha=\Phi^{-1}(1-\alpha)$. Then, the $(1-\alpha)$-quantile of $|\mathbf{G}|_\infty$ can be expressed as $c_\alpha(d) = z_{1/2-(1-\alpha)^{1/d}/2}=\Phi^{-1}(1/2+(1-\alpha)^{1/d}/2)$. The corresponding asymptotic $\alpha$-level {\it Smooth} test is thus defined as
\be
	\Phi^{{\rm S}}_{\alpha}(d)=I\big\{  \hat \Psi(d)  \geq   c_\alpha(d)  \big\} . \label{alpha.level.test}
\ee
The null hypothesis $H_0$ is rejected if and only if $\Phi^{{\rm S}}_{\alpha}(d)=1$.

To construct test that has better power for alternative densities with large energy at high frequencies, we allow the truncation parameter $d=\mbox{dim}(\theta)$ to increase with sample sizes $n$ and $m$. This setup was previously considered by \cite{Fan_1996} in the context of the Gaussian white noise model, where it was argued that if there is a priori evidence that large $\theta_k$'s are located at small $k$, then it is reasonable to select a relatively small $d$; otherwise the resulting test may suffer from low power in detecting densities containing high-frequency components. However, by letting $d$ to increase with sample sizes we allow for different asymptotics than Neyman's fix $d$ large sample scenario. This type of asymptotics aims to illustrate how the truncation parameter $d$ may affect the quality of the test statistic, and to depict a more accurate picture of the behavior for fixed samples. In the present two-sample context, it will be shown (Proposition~\ref{thm.limiting.dist}) that the distribution of $ \hat \Psi(d)$ can still be consistently estimated by that of $|\mathbf{G}|_\infty$ with the truncation parameter $d$ increasing polynomially in $n$ and $m$, where $\mathbf{G}$ is a $d$-dimensional centered Gaussian random vector with covariance matrix $\mathbf{I}_d$. Consequently, the asymptotic size of the smooth test $\Phi^{{\rm S}}_\alpha(d)$ in \eqref{alpha.level.test} coincides with the nominal size $\alpha$ (Theorem~\ref{asymptotic.size}).

\subsection{Choice of the function basis}
\label{basis.sec}

In this paper, we shall focus the following two sets of orthonormal functions with respect to the Lebesgue measure on $[0,1]$, which are the most commonly used  basis for constructing smooth-type goodness-of-fit tests.

\begin{itemize}
\item[(i)](Legendre Polynomial (LP) series). Neyman's original proposal \citep{Neyman_1937} was to use orthonormal polynomials, now known as the normalized Legendre polynomials. Specifically, $\psi_k$ is chosen to be a polynomial of degree $k$ which is orthogonal to all the ones before it and is normalized to size $1$ as in \eqref{orthonormality}. Setting $\psi_0 \equiv 1$, the next four $\psi_k$'s are explicitly given by: $\psi_1(z)=\sqrt{3}(2z-1)$, $\psi_2(z) =\sqrt{5}(6z^2-6z+1)$, $\psi_3(z) =\sqrt{7}(20z^3-30z^2+12z-1)$ and $\psi_4(z) =3(70z^4-140z^3+90z^2-20z+1)$. In general, the normalized Legendre polynomial of order $k$ can be written as
\be
	\psi_k(z)=\frac{\sqrt{2k+1}}{k !} \frac{d^k}{d z^k} (z^2-z)^k, \quad 0\leq z\leq 1, \quad  k=1, 2, \ldots . \label{Legendre.polynomials}
\ee
See, for example, Lemma~1 in \cite{Bera_Ghosh_Xiao_2013}.

\item[(ii)](Trigonometric series). Another widely used basis of orthonormal functions is a trigonometric series given by
\be
	\psi_k(z) = \sqrt{2} \cos(\pi k z), \quad   0\leq z\leq 1, \quad  k= 1, 2, \ldots .  \label{trigonometric.series}
\ee
This particular choice arises in the construction of the weighted quadratic type test statistics, including  the Cram\'er-von Mises and the Anderson-Darling test statistics as prototypical examples \citep{Eubank_LaRiccia_1992,Lehmann_Romano_2005}. Alternatively, one could use the Fourier series which is also a popular trigonometric series given by  $\{ \cos(2\pi kz), \sin(2\pi kz) : k=1, \ldots, d/2 \}$ for $d$ even.
\end{itemize}

Other commonly used compactly supported orthonormal series include spline series \citep{Boor_1978}, Cohen-Deubechies-Vial wavelet series \citep{Mallat_1999} and local polynomial partition series \citep{cattaneo2013optimal}, among others. As the two-sample test statistics constructed in this paper use orthonormal functions that are at least twice continuously differentiable on $[0,1]$, we restrict attention to the Legendre polynomial series \eqref{Legendre.polynomials} and the trigonometric series \eqref{trigonometric.series} only. Indeed, the idea developed here can be directly applied to construct (one-sample) goodness-of-fit tests in one and higher dimensions without imposing smoothness conditions on the series.

\section{Testing equality of two multivariate distributions}
\label{multivariate.sec}

Evidenced by both theoretical (Section~\ref{univariate.theory.sec}) and numerical (Section~\ref{numerical.sec}) studies, we see that Neyman's smooth test principle leads to convenient and powerful tests for univariate data. However, the presence of multivariate joint distributions makes it difficult, or even unrealistic, to consider a direct multivariate extension of the smooth alternatives given in \eqref{alternatives}. In the case of complete independence where all the components of $X$ and $Y$ are independent, the problem for testing equality of two multivariate distributions is equivalent to that for testing equality of many marginal distributions. Neyman's smooth principle can therefore be employed to each of the $p$ marginals.

In this section, we do not impose assumption that limits the dependence among the coordinates in $X$ and $Y$ and note here that the null hypothesis $H_0:F=G$ is equivalent to $H_0: 	u^{\intercal} X   =_d  u^{\intercal} Y, \, \forall u \in \mathcal{S}^{p-1}$. This observation and the idea of projection pursue now allow to apply Neyman's smooth test principle, yielding a family of univariate smooth tests indexed by $u\in \mathcal{S}^{p-1}$ based on which we shall construct our test that incorporates the correlations among all the one-dimensional projections.

\subsection{Test statistics}

Assume that two independent random samples, $X_1,\ldots , X_n$ from the distribution $ F$ and $Y_1,\ldots, Y_m$ from the distribution $G $ are observed, where the two samples sizes are comparable and $m\leq n$. Along every direction $u\in \mathcal{S}^{p-1}$, let $F^u $ and $G^u $ be the distribution functions of one-dimensional projections $u^{\intercal} X$ and $u^{\intercal} Y$, respectively, and define the corresponding empirical distribution functions by
\be
	F^u_{n}(x)=\frac{1}{n}\sn I(u^{\intercal} X_i \leq x) \ \ \mbox{ and } \ \  G^u_{m}(y) = \frac{1}{m} \sm I(u^{\intercal} Y_j\leq y). \label{edf.t}
\ee
As a natural multivariate extension of the KS test, we consider the following test statistic
\be
	\hat{\Psi}_{{\rm MKS}} = \s{\frac{nm}{n+m}} \sup_{(u,t)\in \mathcal{S}^{p-1}\times \br}  | F^u_n(t)-G^u_m(t) |,   \nn 
\ee
which coincides with the KS test when $p=1$. \cite{Baringhaus_Franz_2004} proposed a multivariate extension of the Cram\'er-von Mises test which is of the form
$$
\hat{\Psi}_{{\rm BF}} = \frac{nm}{n+m}\int_{\mathcal{S}^{p-1}}\int_{-\infty}^{+\infty} \{ F^u_n(t) - G^u_m(t)  \}^2 \, dt \, \vartheta(du),
$$
where $\vartheta$ denotes the Lebesgue measure on $\mathcal{S}^{p-1}$. Despite their popularities in practice, the classical omnibus distribution-based testing procedures suffer from low power in detecting fine features such as sharp and short aberrants as well as global features such as high-frequency alternations \citep{Fan_1996}. Now it is well-known that the foregoing drawbacks can be well repaired via smoothing-based test statistics. This motivates the following multivariate smooth test statistic.

As in Section~\ref{method.sec}, let $\{\psi_0\equiv 1, \psi_1, \ldots, \psi_d\}$ ($d\geq 1$) be a set of orthonormal functions and put $\bpsi=(\psi_1,\ldots, \psi_d)^{\intercal}: \br \mapsto \br^d$. Using the union-intersection principle, the two-sample problem of testing $H_0: F = G$ versus $H_1: F\neq G$ can be expressed a collection of univariate testing problems, by noting that $H_0$ and $H_1$ are equivalent to $\cap_{u\in \mathcal{S}^{p-1}} H_{u,0 }$ and $\cup_{u\in \mathcal{S}^{p-1}} H_{u,1 }$, respectively, where
$$
	H_{u,0} : u^{\intercal} X   =_d u^{\intercal} Y , \qquad H_{u,1}: u^{\intercal} X \neq_d  u^{\intercal} Y.
$$
For every marginal null hypothesis $H_{u,0 }$, we consider a smooth-type test statistic in the same spirit as in Section~\ref{oracle.sec} that
\be
	 \Psi_u(d) = \bigg| \frac{1}{m}\sm  \bpsi ( V^u_{j}) \bigg|_\infty   \ \ \mbox{ with } \,  V^u_{j} = F^u(u^{\intercal} Y_j).  \label{marginal.test.statistic}
\ee
For diagnostic purposes, it is interesting to find the best separating direction, i.e. $	u_{\max} : = \arg\max_{u\in \mathcal{S}^{p-1}}   \Psi_u(d)$, along which the two distributions differ most. For the purpose of conducting inference which is the main objective in this paper, we just plug $u_{\max}$ into \eqref{marginal.test.statistic} to get the oracle test statistic ${\Psi}_{\max}(d) := \Psi_{u_{\max}}(d)$, though it is practically infeasible as the distribution function $F $ is unspecified. The most natural and convenient approach is to replace $ V^u_{j}$ in \eqref{marginal.test.statistic} with $\hat{V}^u_j=F^u_n(u^{\intercal} Y_j)$, leading to the following extreme value statistic for testing $H_0:F=G$,
\be
	 \hat{\Psi}_{\max} = \hat{\Psi}_{\max}(d) = \s{\frac{nm}{n+m}} \sup_{u\in \mathcal{S}^{p-1} }\hat{\Psi}_u(d), \label{multivariate.test.statistic}
\ee
where $\hat{\Psi}_u(d) = |m^{-1}\sm \bpsi(\hat{V}^u_j)|_\infty = \max_{1\leq k\leq d}  |  \hat \psi_{u,k} |$ with $ \hat \psi_{u,k} = m^{-1}\sm \psi_k(\hat{V}^u_j)$. Rejection of the null is thus for large values of $\hat{\Psi}_{\max}$, say $\hat{\Psi}_{\max}  > c_\alpha(d)$, where $c_\alpha(d)$ is a critical value to be determined so that the resulting test has the pre-specified significance level $\alpha \in (0,1)$ asymptotically.

\subsection{Critical values}
\label{multivariate.cv.sec}

Due to the highly complex dependence structure among $\{ \hat{\psi}_{u,k}\}_{(u,k)\in \mathcal{S}^{p-1} \times [d]}$, the limiting (null) distribution of $\hat \Psi_{\max}(d)$ may not exist. In fact, $\hat \Psi_{\max}(d)$ can be regarded as the supremum of an empirical process indexed by the class
$$
	\hat{\mathcal{F}}_{n,m} :=\bigg\{  x \mapsto \psi_k \bigg( \frac{1}{n} \sn I \{ u^{\intercal}(x - X_i) \geq 0 \}  \bigg) : (u, k) \in \mathcal{S}^{p-1} \times [d] \bigg\}
$$
of functions $\br^p \mapsto \br$; that is, $\hat \Psi_{\max}(d)= \s{nm/(n+m)}\sup_{f\in \hat{\mathcal{F}}_{n,m}} | m^{-1} \sm f(Y_j)|$. As we allow the dimension $p$ to grow with sample sizes, the ``complexity'' of $\hat{\mathcal{F}}_{n,m}$ increases with $n, m$ and is thus non-Donsker. Therefore, the extreme value statistic $\hat \Psi_{\max}(d)$, even after proper normalization, may not be weakly convergent as $n,m\rightarrow \infty$. Tailored for such non-Donsker classes of functions that change with the sample size, \cite{chernozhukov2013gaussian} developed Gaussian approximations for certain maximum-type statistics under weak regularity conditions. This motivates us to take a different route by using the multiplier (wild) bootstrap method to compute the critical value $ c_\alpha(d)$ for the statistic $\hat \Psi_{\max}(d)$ so that the resulting test has approximately size $\alpha \in (0,1)$.

Let $\{Z_1,\ldots, Z_{N}\}=\{Y_1,\ldots, Y_m, X_1,\ldots, X_n\}$ denote the pooled sample with a total sample size $N=n+m$. For every $(u,k)\in \mathcal{S}^{p-1} \times [d]$, we shall prove in Section~\ref{proof.multivariate.theorem} that
$$
	\sqrt{\frac{nm}{n+m}} \, \hat{\psi}_{u, k} \simeq  \frac{1}{\s{N}} \sum_{j=1}^N w_j \, \psi_k\circ F^u(u^{\intercal} Z_i),
$$
where $w_j=\s{n/m}$ for $j\in [m]$ and $w_j=-\s{m/n}$ for $j\in m+[n]$. This implies that, under certain regularity conditions, $\hat \Psi_{\max}(d)  \simeq  \sup_{ f \in {\mathcal{F}}^p_{ d}}   |N^{-1/2} \sum_{j=1}^N w_j  f(Z_j) |$, where ${\mathcal{F}}^p_{ d}=\{ x \mapsto \psi_k\circ F^u(u^{\intercal}x) : (u,k)\in \mathcal{S}^{p-1}\times [d]\}$. Furthermore, we shall prove in Proposition~\ref{limiting.null.supremum.statistic} that, under the null hypothesis $H_0:F=G$,
\be
	 \hat \Psi_{\max}(d)     \simeq_d   \,  \| \bG   \|_{ {\mathcal{F}}^p_{ d}} := \sup_{ f \in {\mathcal{F}}^p_{ d}}  | \bG f  |, \label{heuristic.approximation}
\ee
where $\bG $ is a centered Gaussian process indexed by ${ \mathcal{F}}^p_{ d}$ with covariance function
\begin{align}
 &    	E \{ ( \bG    f_{u, k} ) ( \bG   f_{v, \ell} ) \}  \nn \\
 &  = P_X ( f_{u, k}  f_{v,\ell} )  = E  \{ f_{u,k} (X)   f_{v,\ell}(X) \}   =\int_{\br^p}  \psi_k(F^u(u^{\intercal} x))\psi_\ell(F^v(v^{\intercal} x)) \, dF(x),  \label{sigma.ukvl}
\end{align}
where $U^u =F^u(u^{\intercal} X) =_d {\rm Unif}(0,1)$ for $u\in \mathcal{S}^{p-1}$. In particular, $E \{ ( \bG  f_{u, k} ) ( \bG f_{u, \ell} ) \}= \delta_{k \ell}$. The distribution of $ \| \bG \|_{{\mathcal{F}}^p_{ d}}$, however, is unspecified because its covariance function is unknown. Therefore, in practice we need to replace it with a suitable estimator, and then simulate the Gaussian process $\bG$ to compute the critical value $c_\alpha(d)$ numerically, as described~below.

\medskip
\noi
{\sc Multiplier bootstrap}.
\begin{itemize}
\item[(i)] Independent of the observed data $\{X_i\}_{i=1}^n$ and $\{Y_j\}_{j=1}^m$, generate i.i.d. standard normal random variables $e_1,\ldots, e_n$. Then construct the {\it Multiplier Bootstrap} statistic
\be
	\hat{\Psi}_{\max}^{{\rm MB}} = \hat{\Psi}_{\max}^{{\rm MB}}(d) =\sup_{(u,k)\in \mathcal{S}^{p-1} \times [d]} \bigg| \frac{1}{\s{n}}\sn e_i \psi_k(\hat{U}^u_i) \bigg|,  \label{bootstrap.statistic}
\ee
where $\hat{U}^u_i=F^u_n(u^{\intercal} X_i)$ for $F^u_n(\cdot)$ as in \eqref{edf.t}.

\item[(ii)] Calculate the data-driven critical value $\hat c^{{\rm MB}}_\alpha(d)$ which is defined as the conditional $(1-\alpha)$-quantile of $\hat{\Psi}_{\max}^{{\rm MB}}$ given $\{X_i\}_{i=1}^n$; that is,
\be
	\hat c^{{\rm MB}}_\alpha( d) = \inf\big\{ t\in \br :  P_e \big( \hat{\Psi}_{\max}^{{\rm MB}} > t \big)  \leq \alpha \big\} \label{data-driven.critical.value}
\ee
where $P_e$ denotes the probability measure induced by the normal random variables $\{e_i\}_{i=1}^n$ conditional on $\{X_i\}_{i=1}^n$.
\end{itemize}

For every $t\geq 0$, $P_e(\hat{\Psi}_{\max}^{{\rm MB}} \leq t)$ is a random variable depending on $\{X_i\}_{i=1}^n$ and so is $\hat{c}^{{\rm MB}}_\alpha(d)$, which can be computed with arbitrary accuracy via Monte Carlo simulations. Consequently, we propose the following {\it Multivariate Smooth} test
\be
	 {\Phi}^{{\rm MS}}_{\alpha}(d) = I \big\{ \hat{\Psi}_{\max}(d) \geq \hat c^{{\rm MB}}_\alpha( d) \big\} .  \label{multivariate.test}
\ee
The null hypothesis $H_0: F=G$ is rejected if and only if ${\Phi}^{{\rm MS}}_{\alpha}(d) =1$.

\section{Theoretical properties}
\label{theory.section}

Assume that we are given independent samples from the two (univariate and multivariate) distributions. As different sample sizes are allowed, for technical reasons we need to impose assumptions about the way in which samples sizes grow. The following gives the basic assumptions on the sampling process.

\begin{assumption} \label{basic.assumption} \ \
{\rm
\begin{itemize}
\item[(i)] $\{X_1,\ldots, X_n\}$ and $\{Y_1,\ldots, Y_m\}$ are two independent random samples from $X $ and $Y $, with absolute continuous distribution functions $F $ and $G $, respectively;

\item[(ii)] The sample sizes, $n$ and $m$, are comparable in the sense that $c_0 n \leq m\leq n$ for some constant $0<c_0\leq 1$.
\end{itemize}
}
\end{assumption}

Let $\{\psi_0\equiv 1, \psi_1, \ldots, \psi_d\}$ be a sequence of twice differentiable orthonormal functions $[0,1] \mapsto \br$, where $d\geq 1$ is the truncation parameter. Moreover, for $\ell=0,1 , 2$, define
\be
 B_{\ell d} =  \max_{1\leq k\leq d} \| \psi_k^{(\ell)} \|_\infty = \max_{1\leq k\leq d} \max_{z \in [0,1]}  |\psi^{(\ell)}_k(z)| . \label{Bd.def}
\ee
These quantities will play a key role in our analysis. For the particular choice of the function basis as in \eqref{Legendre.polynomials} and \eqref{trigonometric.series}, we specify below the order of $B_{\ell d}$, as a function of $d$, for $\ell=0, 1, 2$.

\begin{itemize}
\item[(i)](Legendre polynomial series). For the normalized Legendre polynomials $\psi_k$, it is known that $B_{0d}=\max_{1\leq k\leq d} \max_{z\in [0,1]} |\psi_k(z)| = \sqrt{2d+1}$. See, e.g. \cite{Sansone_1959}. Moreover, by the Markov inequality \citep{Shadrin_1992}, $\|  \psi'_k \|_\infty \leq k^2 \|  \psi_k \|_\infty$ and $\|  \psi''_k \|_\infty \leq \frac{k^2(k^2-1)}{3} \|  \psi_k \|_\infty$. Together with \eqref{Bd.def}, this implies
\be
	B_{0d} = \sqrt{2d+1} , \quad B_{1d} \leq  \sqrt{3} \, d^{\,5/2}  \quad \mbox{ and } \quad  B_{2d} \leq  3^{-1/2} d^{\,9/2}.  \label{LP.bound}
\ee

\item[(ii)](Trigonometric series). For the trigonometric series $\psi_{k}(z)=\sqrt{2} \cos(\pi kz)$, it is straightforward to see that $\psi'_k(z)=-\sqrt{2}\, \pi k \sin(\pi kz)$ and $\psi''_k(z)=-\sqrt{2}\,\pi^2 k^2 \cos(\pi kz)$. Consequently, we have
\be
	B_{0d} = \sqrt{2} , \quad B_{1d} \leq  \sqrt{2} \,  \pi d  \quad \mbox{ and } \quad  B_{2d} \leq   \sqrt{2}\, \pi^2 d^2.  \label{T.bound}
\ee
\end{itemize}

\subsection{Asymptotic properties of $\Phi^{{\rm S}}_\alpha(d)$}
\label{univariate.theory.sec}

\begin{assumption} \label{regularity.assumption}
{\rm
The truncation parameter $d$ is such that $d\leq m$ and as $n\rightarrow \infty$,
$$
	 (\log n)^{7/6} B_{0d} =o(n^{1/6})  , \quad   ( \log n)^{3/2} B_{1 d} =o(n^{1/2}  ), \quad  (\log n)^{1/2} B_{2d} =o(n^{1/2}) .
$$
}
\end{assumption}

The next theorem establishes the validity of the univariate smooth test $\Phi^{{\rm S}}_\alpha(d)$ in \eqref{alpha.level.test}.

\begin{theorem} \label{asymptotic.size}
Suppose that Assumptions~\ref{basic.assumption} and \ref{regularity.assumption} hold. Then as $n, m \rightarrow \infty$,
\be
	 \sup_{0<\alpha< 1} \big|  P_{ H_{0}} \big\{ \Phi^{{\rm S}}_{ \alpha}(d) =1 \big\} - \alpha \big| \rightarrow 0.	\label{size.consist}
\ee
\end{theorem}

\begin{remark}
{\rm
In view of \eqref{LP.bound} and \eqref{T.bound}, it follows from Theorem~\ref{asymptotic.size} that the error in size of the smooth test $\Phi^{{\rm S}}_\alpha(d)$ using the trigonometric series \eqref{trigonometric.series} (resp. Legendre polynomials series \eqref{Legendre.polynomials}) tends to zero provided that $d=o\{(n/\log n)^{1/4}\}$ (resp. $d=o\{(n/\log n)^{1/9}\}$) as $n \to \infty$. }
\end{remark}

Next we consider the asymptotic power of $\Phi^{{\rm S}}_{\alpha}(d)$ against local alternatives when $d=d_{n,m}\rightarrow \infty$ as $n, m\rightarrow \infty$. For the following results, let $\bar{n}=2( n^{-1}+ m^{-1})^{-1} $ denote the {\it harmonic mean} of the two sample sizes. Our oracle statistic $\Psi_{{\rm S}}$ given in \eqref{oracle.test} mimics $\s{ \bar{n}/2} \, \max_{1\leq k\leq d} |E_\theta\{\psi_k(V)\} |$,
where $\theta= (\theta_1, \ldots, \theta_d)^{\intercal} \in \br^d$. Consider testing $\tilde{H}_0: \rho \equiv 0$ in \eqref{equivalent.null} against the following local alternatives
\be
	H^d_{1} :  \rho = \rho_{\theta} ,  \, \mbox{ for } \theta \in \Theta := \bigg\{ b=(b_1,\ldots, b_d)^{\intercal} \in \br^d: \max_{1\leq k\leq d }|b_k | = \la \s{\frac{\log d}{\bar{n}}} \bigg\},  \label{local.alternative}
\ee
where $\rho_\theta$ is as in \eqref{alternatives} and $\la >0$ is a separation parameter. It is clear that the difficulty of testing between $\tilde{H}_0$ and $H^d_{1}$ depends on the value of $\la$; that is, the smaller $\la$ is, the harder it is to distinguish between the two hypotheses. The power of the test $\Phi^{{\rm S}}_\alpha(d)$ in \eqref{alpha.level.test} against $H^d_{1}$ is provided by the following theorem.

\begin{theorem} \label{asymptotic.power.1}
Suppose that Assumption~\ref{basic.assumption} holds. The truncation parameter $d=d_{n,m}$ is such that $d=o(n^{1/4})$ if the trigonometric series \eqref{trigonometric.series} is used to construct the test statistic $\hat{\Psi}(d)$ in \eqref{data-driven.test} and $d=o(n^{1/9})$ if the Legendre polynomials series \eqref{Legendre.polynomials} is used. Then, under $H^d_{1}$ with $\la \geq 2 +\varepsilon$ for some $\varepsilon>0$,
\be
	\lim_{n,d\rightarrow \infty} P_{  H^d_{1}}\big\{ \Phi_{\alpha}^{{\rm S}}(d) =1 \big\}  =1 .	\label{power.consist}
\ee
\end{theorem}

\subsection{Asymptotic properties of $\Phi^{{\rm MS}}_\alpha(d)$}
\label{multivariate.theory.sec}

In this section, we consider the multivariate case where the dimension $p=p_{n,m}$ is allowed to grow with sample sizes, and hence our results hold naturally for the fix dimension scenario. Specifically, we impose the following assumption on the quadruplet $(n,m,p,d)$.

\begin{assumption} \label{regularity.assumption.2}
{\rm
There exist constants $C_0, C_1>0$ and $c_1 \in (0, 1)$ such that
\be
  d\leq \min\{n, m, \exp(C_0 \, p)\} , \quad  \max \big( p^{7} B_{1d}^2  , p B_{2d}^2   \big) \leq C_1 \, n^{1-c_1}. \label{constraint.pd}
\ee
}
\end{assumption}

The next theorem establishes the validity of the multivariate smooth test $\Phi^{{\rm MS}}_\alpha(d)$.

\begin{theorem} \label{multivariate.asymptotic.size}
Suppose that Assumptions~\ref{basic.assumption} and \ref{regularity.assumption.2} hold. Then as $n, m \rightarrow \infty$,
\be
	 \sup_{0<\alpha< 1} \big|  P_{ H_{0}} \big\{ \Phi^{{\rm MS}}_{ \alpha}(d) =1 \big\} - \alpha \big| \rightarrow 0.	\label{size.consist}
\ee
\end{theorem}

\section{Numerical studies}   \label{numerical.sec}

In this section, we illustrate the finite sample performance of the proposed smooth tests described in Sections~\ref{method.sec} and \ref{multivariate.sec} via Monte Carlo simulations. The univariate and the multivariate cases will be studied separately.

\subsection{Univariate case}

Proposition~\ref{thm.limiting.dist} in Section~\ref{proof.section} shows that the distribution of the test statistic $\hat{\Psi}(d)$ in \eqref{data-driven.test} can be consistently estimated by that of the absolute Gaussian maximum $|\mathbf{G}|_\infty$, where $\mathbf{G} =_d N(0, \mathbf{I}_d) $. To see how close this approximation is, we compare in Figure~\ref{size.figure} the cumulative distribution function of $|\mathbf{G}|_\infty$ and the empirical distributions of $\hat{\Psi}(d)$ using the trigonometric series \eqref{trigonometric.series} and the Legendre polynomial (LP) series \eqref{Legendre.polynomials}, when the data are generated from Student's $t(7)$-distribution, with $n=180, m=150$ and $d=12$. We only present the upper half of the curve since the $(1-\alpha)$ quantile of $|\mathbf{G}|_\infty$ with $\alpha\in (0,1/2)$ is of particular interest. It can be seen from Figure~\ref{size.figure} that the cumulative distribution curves of $|\mathbf{G}|_\infty$ and the trigonometric series based statistic, denoted by T-$\hat{\Psi}(d)$, almost coincide, while there is a slightly noticeable difference between $|\mathbf{G}|_\infty$ and the LP polynomial series based statistic LP-$\hat{\Psi}(d)$. Indeed, this phenomenon can be expected from the theoretical discoveries. See, for example, the rate of convergence in \eqref{c123n} and \eqref{gassian.approxi.1}, where dependence of $\{B_{\ell d} \}_{ \ell=0, 1, 2}$ on $d$ can be found in \eqref{LP.bound} and \eqref{T.bound}.

\begin{figure}[h]
\centering
\caption{ Comparison of the empirical cumulative distributions of LP-$\hat{\Psi}(12)$, T-$\hat{\Psi}(12)$ and the limiting cumulative distribution with $n=180$ and $m=150$. The plot is based on 5000 simulations.}
\label{size.figure}
\includegraphics[trim = 35mm 85mm 35mm 90mm, clip,width=5.2in,height=8cm]{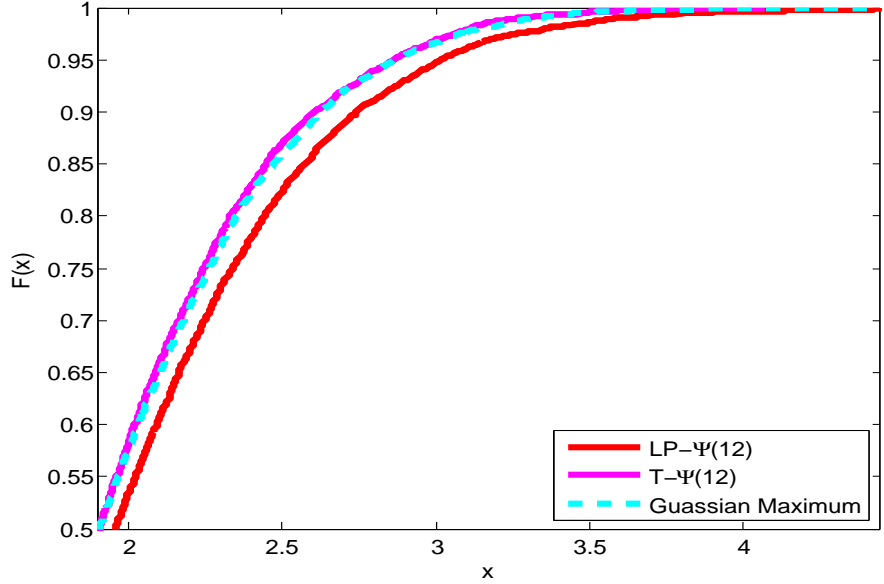}
\end{figure}

Next, we carry out 5000 simulations with nominal significance level $\alpha=0.05$ to calculate the empirical sizes of the proposed smooth test $\Phi_\alpha^{{\rm S}}(d)$. We denote with T-$\Phi_\alpha^{{\rm S}}(d)$ and LP-$\Phi_\alpha^{{\rm S}}(d)$, respectively, the tests based on the trigonometric series \eqref{trigonometric.series} and the LP polynomial series \eqref{Legendre.polynomials}. The sample sizes $(n,m)$ are taken to be (80,60), (120,90), (180,150), and $d$ takes values $4, 8, 12$. We compare the proposed smooth test with the testing procedure proposed by \cite{Bera_Ghosh_Xiao_2013}, the two-sample Kolmogorov-Smirnov test and the two-sample Cram\'er-von Mises test in five examples when the data are generated from Gamma, Logistic, Gaussian, Pareto and Stable distributions. The results are summarized in Table \ref{size.table}, from which we see that among all the five examples considered, the empirical sizes of T-$\Phi^{{\rm S}}_\alpha(d)$ with $d\in \{ 4,8, 12\}$ are close to 0.05. This highlights the robustness of the testing procedure T-$\Phi^{{\rm S}}_\alpha(d)$ with respect to the choice of the truncation parameter $d$. Further, we note that the empirical sizes of LP-$\Phi^{{\rm S}}_\alpha(4)$ are comparable to those of $\Phi_{{\rm BGX}}$, while as $d$ increases, the test LP-$\Phi^{{\rm S}}_\alpha(d)$ suffers from size distortion gradually. In fact, as pointed out by \cite{Neyman_1937} and \cite{Bera_Ghosh_2002}, when the Legendre polynomials series is used to construct the test statistic, the effectiveness of the corresponding test in each direction could be diluted if $d$ is too large. Nevertheless, the test based on the trigonometric series remains to be efficient as $d$ increases and can be very powerful as we shall see later.

\begin{table}[h]
\centering
\caption{\label{size}Comparison of empirical sizes with nominal significance level $\alpha = 0.05$}
\label{size.table}
\footnotesize{\begin{tabular}{llcccccccccc}
\toprule
&~~~&\multicolumn{3}{c}{$\text{T-}{\Phi}_{\alpha}^{\text{S}}(d)$}&&\multicolumn{3}{c}{$\text{LP-}{\Phi}_{\alpha}^{\text{S}}(d)$}&BGX&KS&CVM \\

\cline{3-5}  \cline{7-9}\vspace{-0.3cm}\\
Model&$(n,m)$&$d=4$&$d=8$&$d=12$&&$d=4$&$d=8$&$d=12$&\\
\midrule
Gamma(2,2)&$(80,60)$&0.0504 &0.0500 &0.0490 	&&0.0584 &0.0724 &0.1078 &0.0634 &0.0494 &0.0524 \\

 &$(120,90)$&0.0504 &0.0510 &0.0484 && 0.0542 &0.0654 &0.0830 &0.0590 &0.0438	&0.0434		
\\
 &$(180,150)$&0.0496	&0.0486	&0.0484		&&0.0500	&0.0554	&0.0706	&0.0530	&0.0440	&0.0444\\
\midrule
\midrule
Logistic(0,1)&$(80,60)$&0.0498 	&0.0496 	&0.0482 	&&0.0576 	&0.0748	&0.1038	&0.0618	&0.0456&	0.0494
\\
 &$(120,90)$&0.0504 	&0.0500 	&0.0496 		&&0.0528 	&0.0666 	&0.0860 	&0.0552 	&0.0504	&0.0466
\\
 &$(180,150)$&0.0502	&0.0498	&0.0500		&&0.0508	&0.0570	&0.0696	&0.0574	&0.0438	&0.0424
\\
\midrule
\midrule
N(0,1)&$(80,60)$&0.0504 	&0.0488 	&0.0470 		&&0.0570 	&0.0764 	&0.1060 	&0.0648 	&0.0494	 &0.0504
\\
 &$(120,90)$&0.0502 	&0.0494 	&0.0482 		&&0.0548 	&0.0694 	&0.0850 	&0.0566 	&0.053	&0.0504
\\
 &$(180,150)$&0.0502	&0.0514	&0.0516	&&0.0504	&0.0544	 &0.0616	&0.0542	&0.0446	&0.0500
\\
\midrule
\midrule
 Pareto(0.5,1,1)&$(80,60)$&0.0502 &0.0484 &0.0468 	&&0.0582 &0.0766 	&0.1064 	&0.0640 	&0.0460   &0.0494
\\
 &$(120,90)$&0.0500 	&0.0494 	&0.0494  &&0.0540 	&0.0640 	&0.0824 	&0.0592 	&0.0468	&0.0480
\\
 &$(180,150)$&0.0498	&0.0496	&0.0500		&&0.0530	&0.0586	&0.0724	&0.0542	&0.0436	&0.0456
\\
\midrule
\midrule
 Stable(1.5,0,1,1)&$(80,60)$&0.0480 	&0.0470 	&0.0456 	&&0.0544 	&0.0758 	&0.1088 	&0.0606	&0.0474	&0.0488
\\
 &$(120,90)$&0.0496 	&0.0498 	&0.0494 		&&0.0578 	&0.0692 	&0.0790 	&0.0614 	&0.0492	 &0.0514
\\
 &$(180,150)$&0.0510  &0.0514	&0.0506		&&0.0508	&0.0570	 &0.0690	&0.0608	 &0.0490 &0.0484
\\
\bottomrule
\end{tabular}}
\end{table}\par

The power performance is evaluated through the following five examples. In each example, the result reported is based on 1000 simulations where samples sizes $(n,m)$ are taken to be $(120,90)$ and $(180,150)$. Because of the distortion of empirical sizes of LP-$\Phi_\alpha^{{\rm S}}(d)$, we only compare the power of the trigonometric series based smooth test T-$\Phi_\alpha^{{\rm S}}(d)$ with that of the KS, CVM and BGX tests. The plots of power functions against different families of alternative distributions from Examples \ref{eg1}--\ref{eg5} are given in Figure \ref{Power-compare}.

\begin{Example}\label{eg1}
\begin{align*}
& X: F =\text{uniform}\,(-1,1) \ \ \mbox{versus } \\
  & Y: G =G_\mu \ \ \mbox{ with density } \ \ g_\mu(x) =\frac{1}{2}+2x \frac{\mu-|x|}{\mu^2}  I(|x|<\mu) \quad (0 \le \mu\le 1).
\end{align*}
\end{Example}

\begin{Example}\label{eg2}
\begin{align*}
& X: F =\text{uniform}\,(-1,1) \ \ \mbox{ versus}  \\
& Y: G =G_\sigma \ \ \mbox{ with density }  \ \  g_\sigma(x) =\frac{1}{2}\{1+\sin(2\pi\sigma x)\} \quad (0.5\le \sigma \le 5).
\end{align*}
\end{Example}

\begin{Example}\label{eg3}
\begin{align*}
& X: F  =\text{lognormal}\,(0,1) \ \ \mbox{ with density} \ \  f(x)=(2\pi)^{-1/2}x^{-1}\exp\{-(\log x)^2/2\}\ \ \mbox{ versus } \\
   & Y: G =G_a  \ \ \mbox{ with density} \ \  g_a(x) =f(x)\{1+a\sin(2\pi \log x)\} \ \  (-1\le a \le 1).
\end{align*}
\end{Example}

\begin{Example}\label{eg4}
\begin{align*}
& X: F =\text{uniform}\,(0,1)  \ \  \mbox{ versus } \\
& Y: G =G_c \ \ \mbox{ with density } \ \  g_c(x) =\exp\{ c\sin(5\pi x)\} \ \ (0\le c \le 2).
\end{align*}
\end{Example}

\begin{Example}\label{eg5}
\begin{align*}
& X: F =\text{uniform}\,(0,1)	\ \  \mbox{ versus } \\
& Y: G =G_c \ \ \mbox{ with density } \ \   g_c(x) =1+c\cos(5\pi x) \ \  (0\le c \le 2).
\end{align*}
\end{Example}

The first two examples are, respectively, Examples 5 and 6 in \cite{Fan_1996} which were designed to demonstrate the performance of the adaptive Neyman's test proposed there. In Example~\ref{eg1}, when $\mu=0$, $G_\mu$ coincides with $F$. For this family of alternatives index by $\mu$, the strength of the local feature depends on $\mu$ in the sense that the larger the $\mu$, the stronger the local feature. As expected, the powers of all the tests considered grow with $\mu$ and when sample sizes are large enough, the smooth tests T-$\Phi^{{\rm S}}_\alpha(d)$ uniformly outperform the others. Example~\ref{eg2}, on the other hand, is designed to test the global features with various frequencies. It can be seen from the second row in Figure~\ref{Power-compare} that the test T-$\Phi^{{\rm S}}_\alpha(16)$ has the highest power that approaches to 1 rapidly as $\sigma$ decreases to 0. The third example is from \cite{Heyde_1963}, where $g_a$ is a density and has the same moments as $f_0$ of any order. In this example, all the KS, CVM and BGX tests suffer from very poor power, while surprisingly, the smooth tests based on the trigonometric series remain powerful. The last two examples aim to cover the high-frequency alternations. Again, the proposed tests have the highest powers. In fact, the BGX test was originally constructed to identify deviations in the directions of mean, variance, skewness and kurtosis, and hence it can be relatively less powerful in detecting local features or high-frequency components.


\begin{figure}[!htbp]
\centering
\caption{Empirical powers for Examples~\ref{eg1}--\ref{eg5} based on 1000 replications with $\alpha=0.05$}
\label{Power-compare}
\begin{tabular}{cc}
\includegraphics[trim = 35mm 85mm 35mm 80mm, clip,width=2.8in,height=4.2cm]{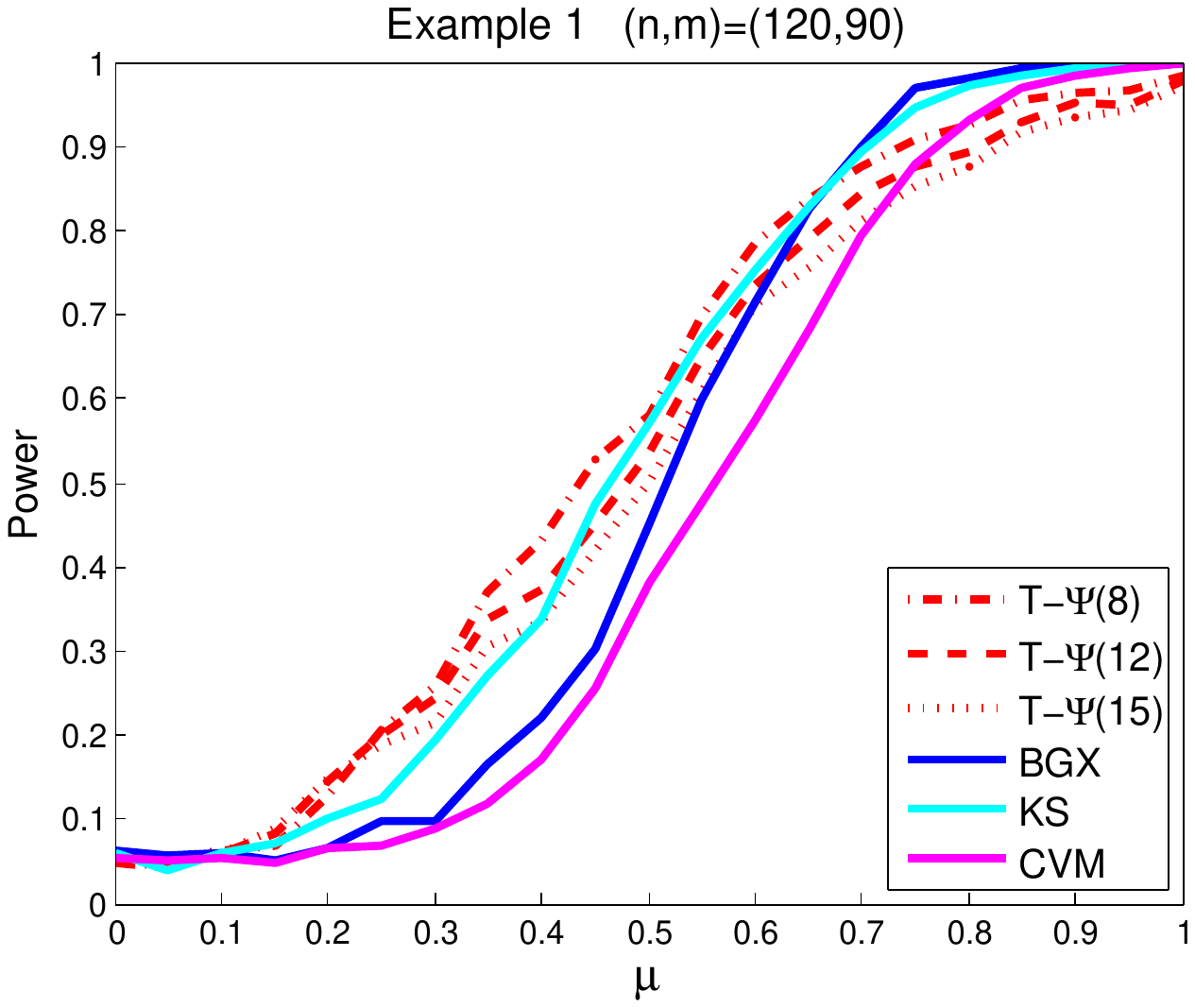}
&\includegraphics[trim = 35mm 85mm 35mm 80mm, clip,width=2.8in,height=4.2cm]{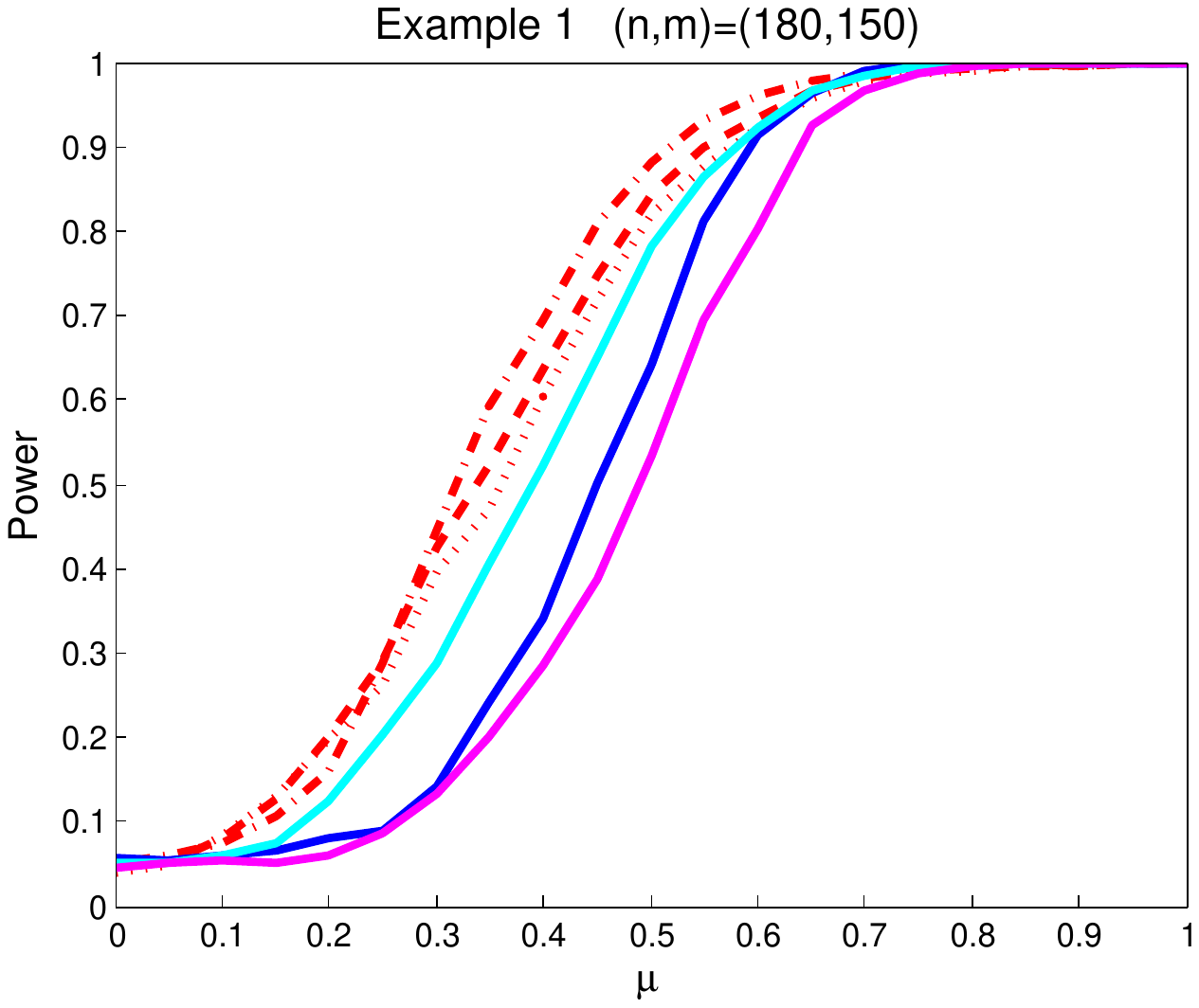}\\
\includegraphics[trim = 35mm 85mm 35mm 80mm, clip,width=2.8in,height=4.2cm]{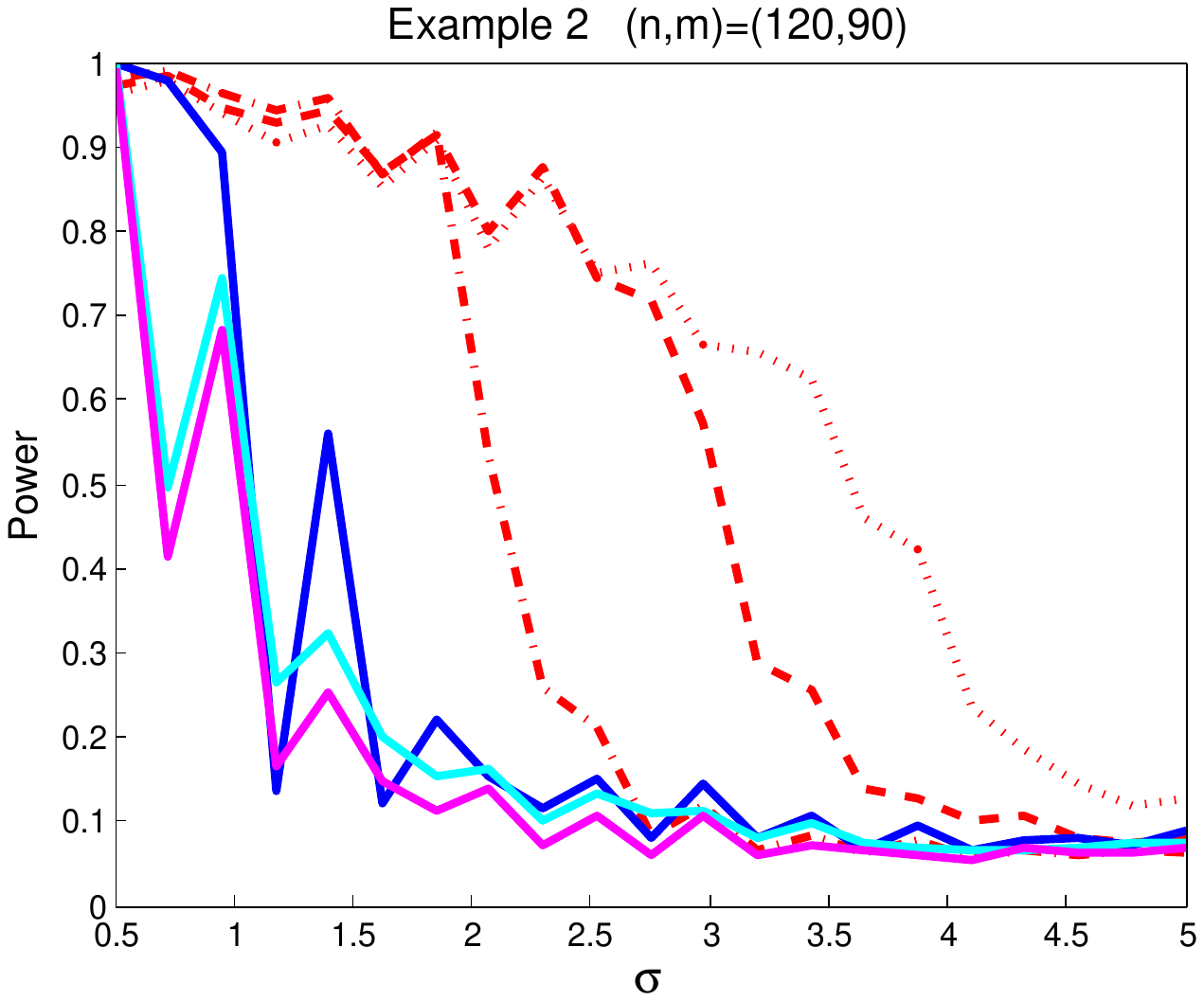}
&\includegraphics[trim = 35mm 85mm 35mm 80mm, clip,width=2.8in,height=4.2cm]{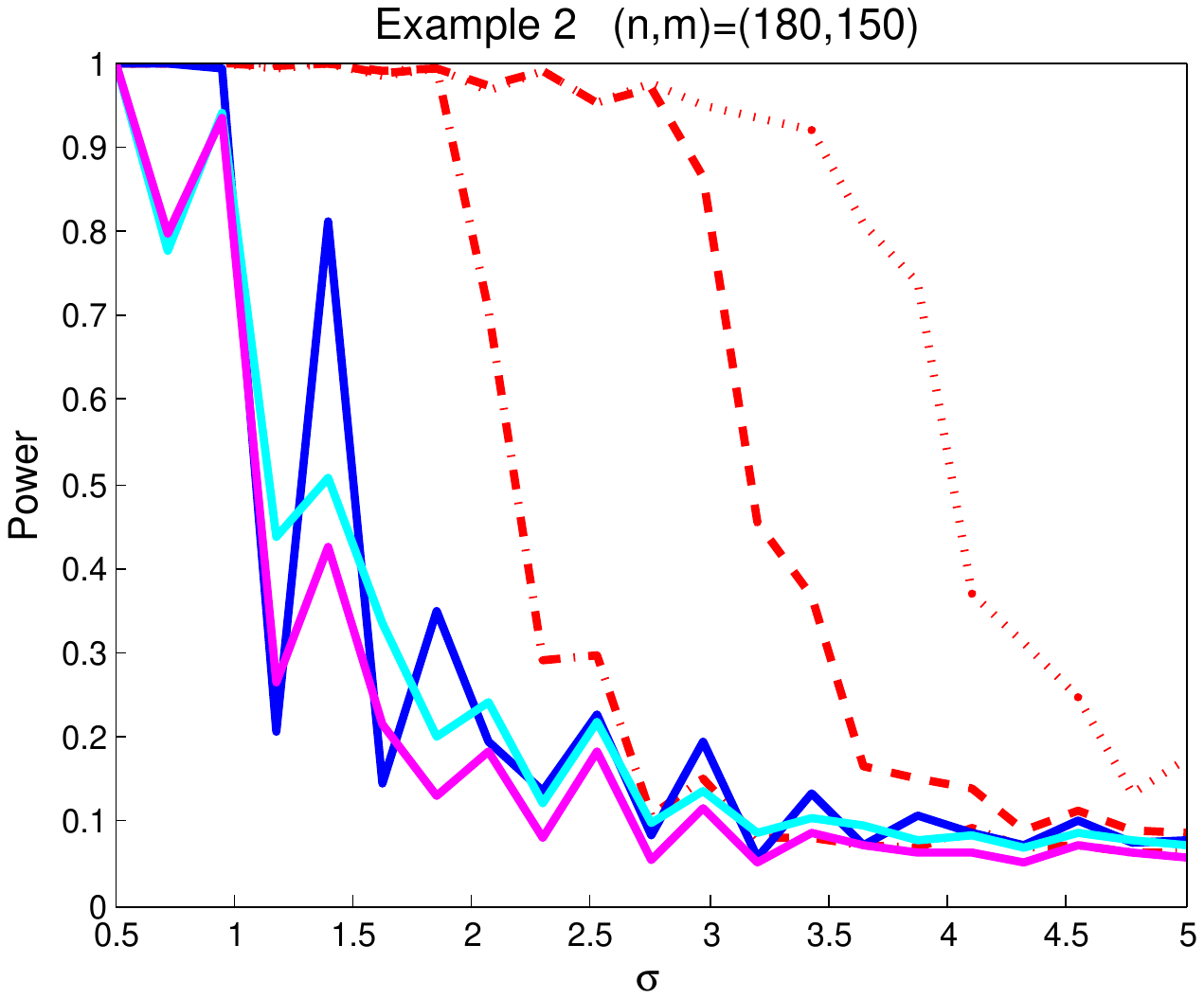}\\
\includegraphics[trim = 35mm 85mm 35mm 80mm, clip,width=2.8in,height=4.2cm]{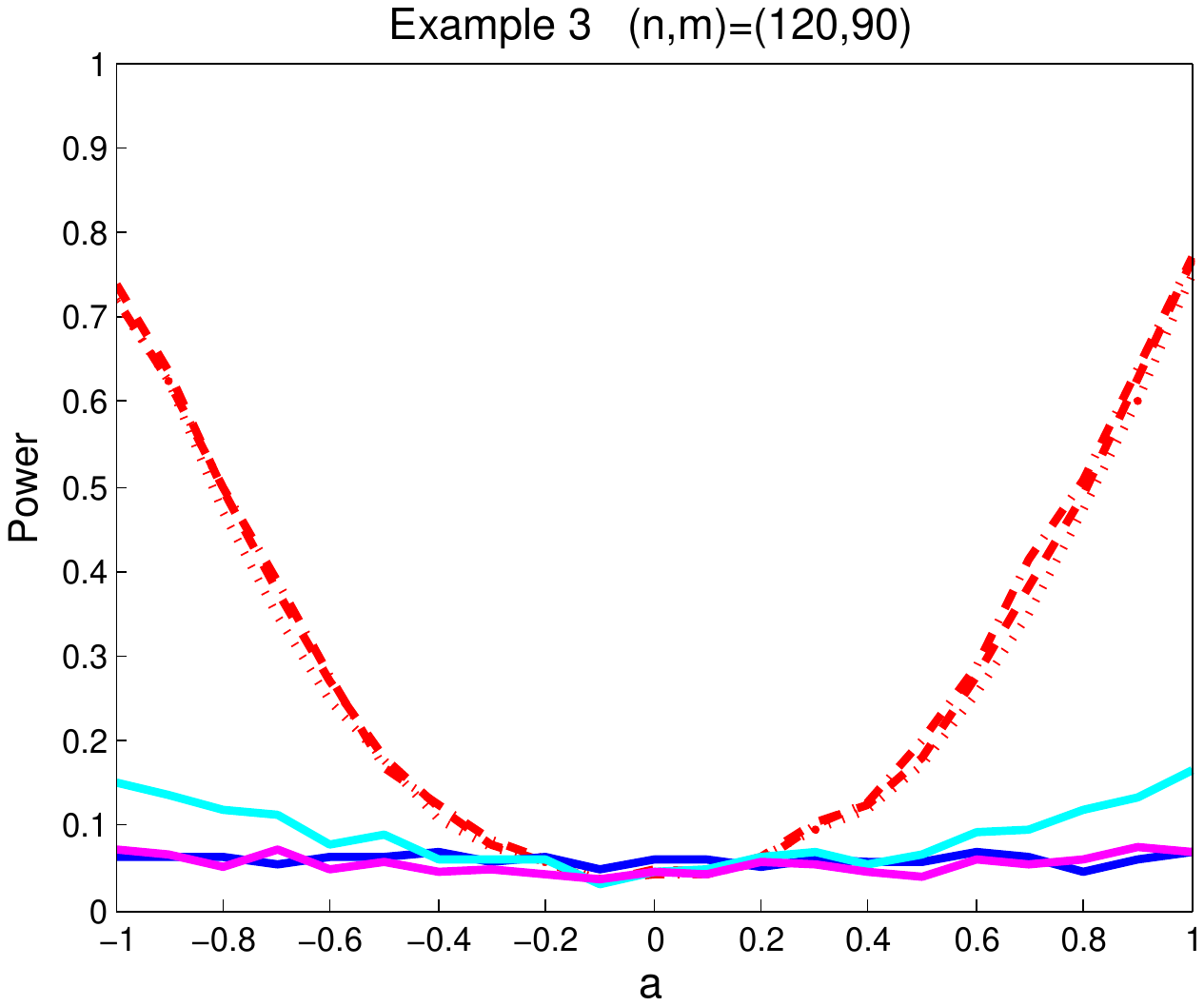}
&\includegraphics[trim = 35mm 85mm 35mm 80mm, clip,width=2.8in,height=4.2cm]{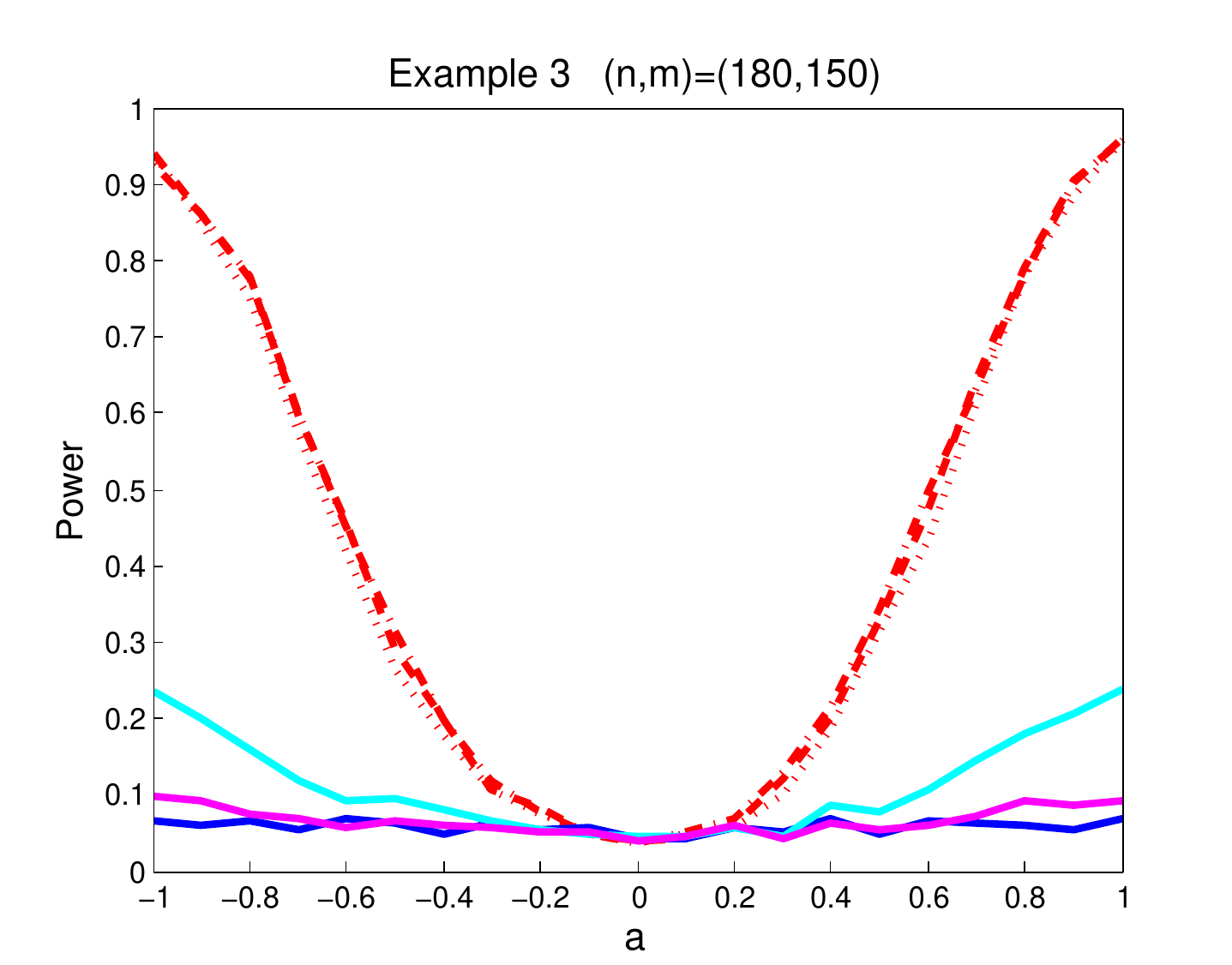}\\
\includegraphics[trim = 35mm 85mm 35mm 80mm, clip,width=2.8in,height=4.2cm]{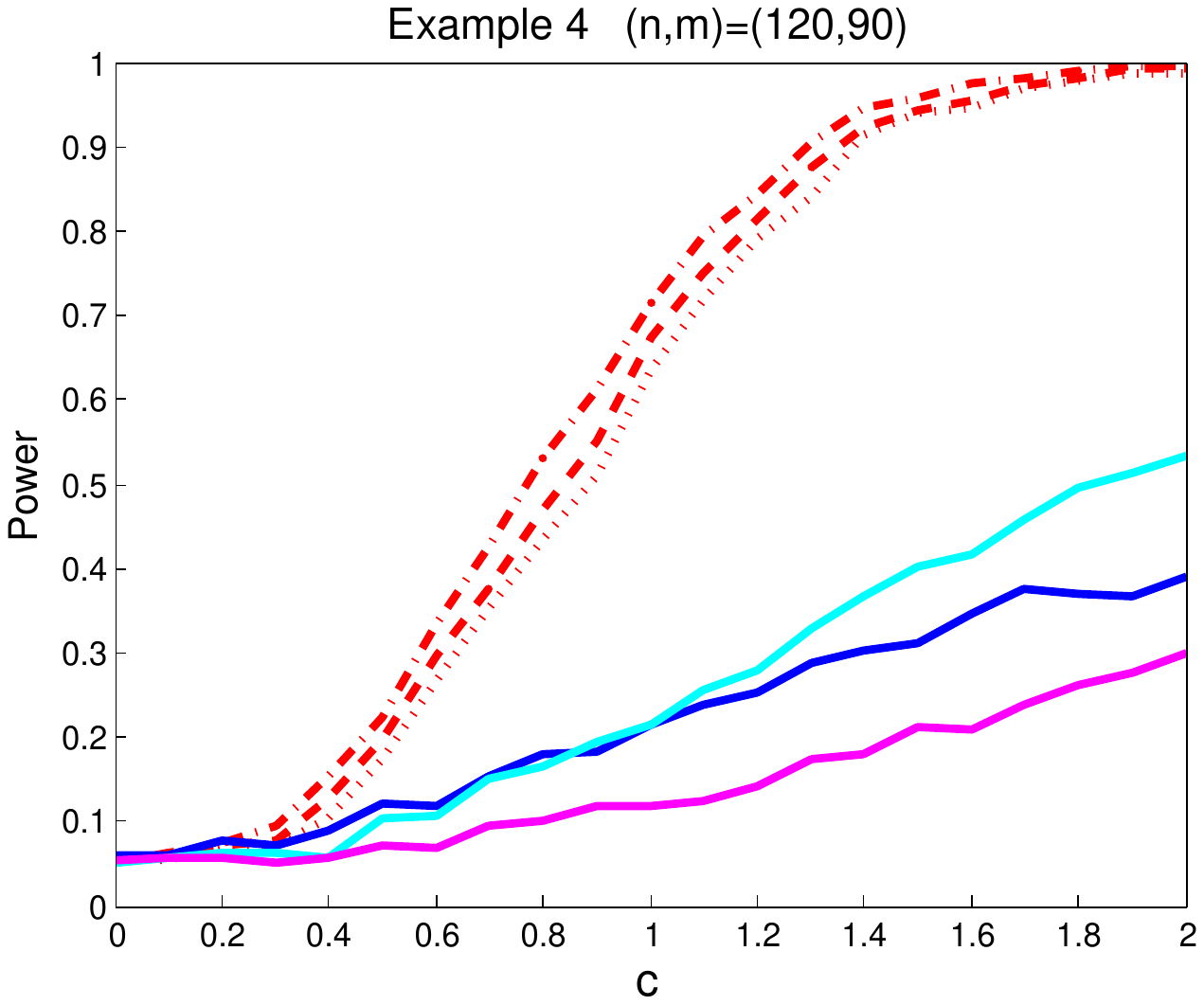}
&\includegraphics[trim = 35mm 85mm 35mm 80mm, clip,width=2.8in,height=4.2cm]{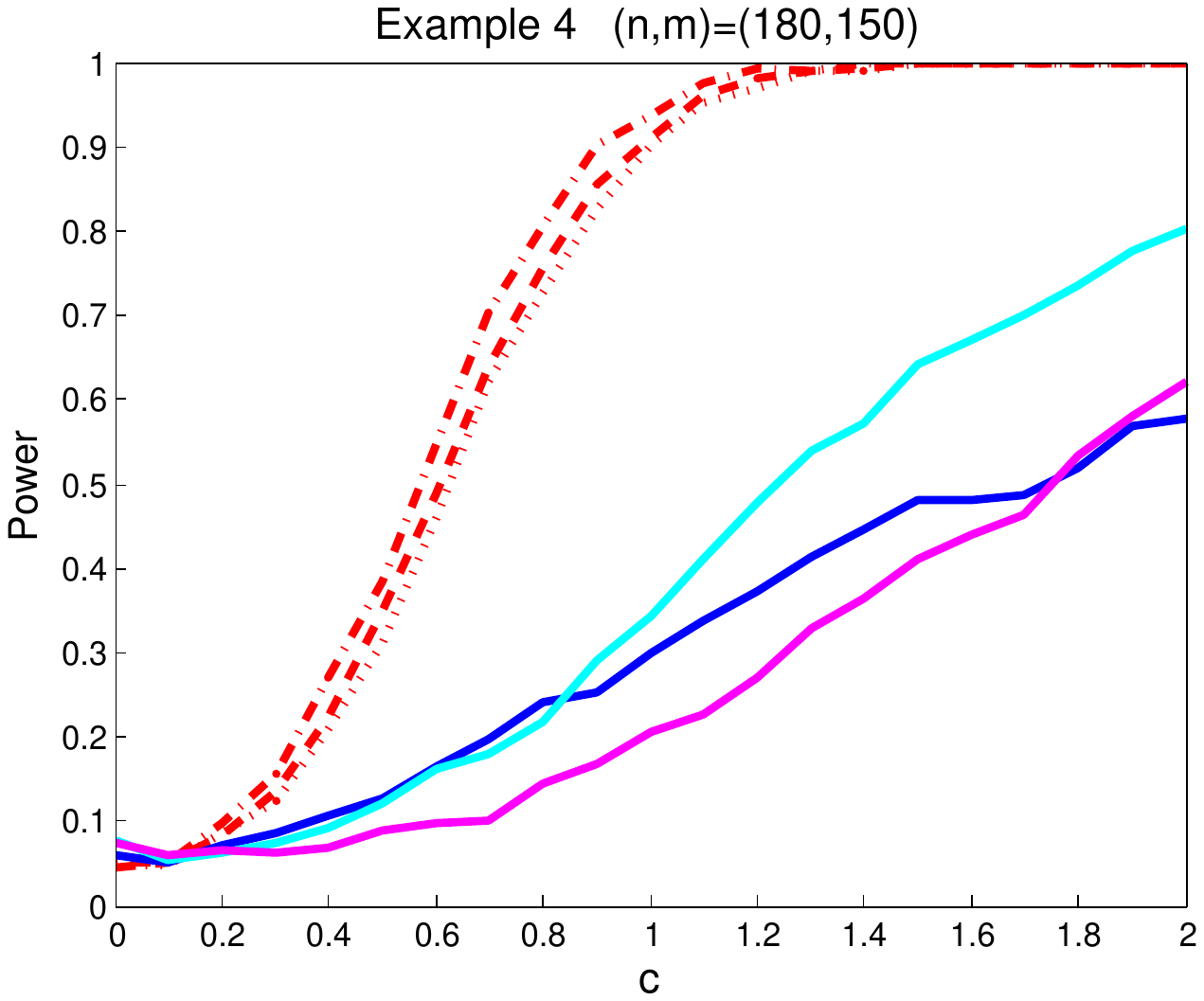}\\
\includegraphics[trim = 35mm 85mm 35mm 80mm, clip,width=2.8in,height=4.2cm]{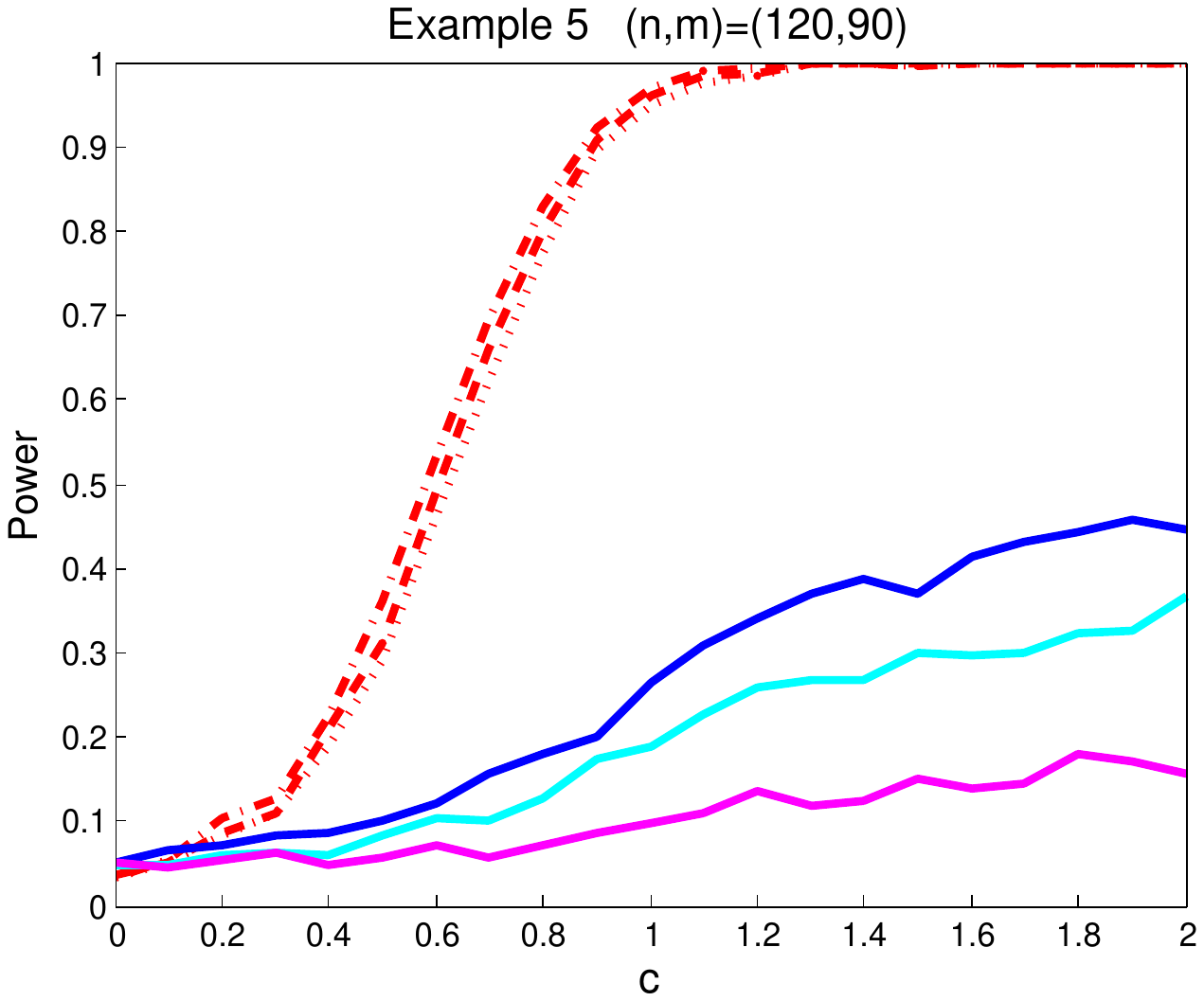}
&\includegraphics[trim = 35mm 85mm 35mm 80mm, clip,width=2.8in,height=4.2cm]{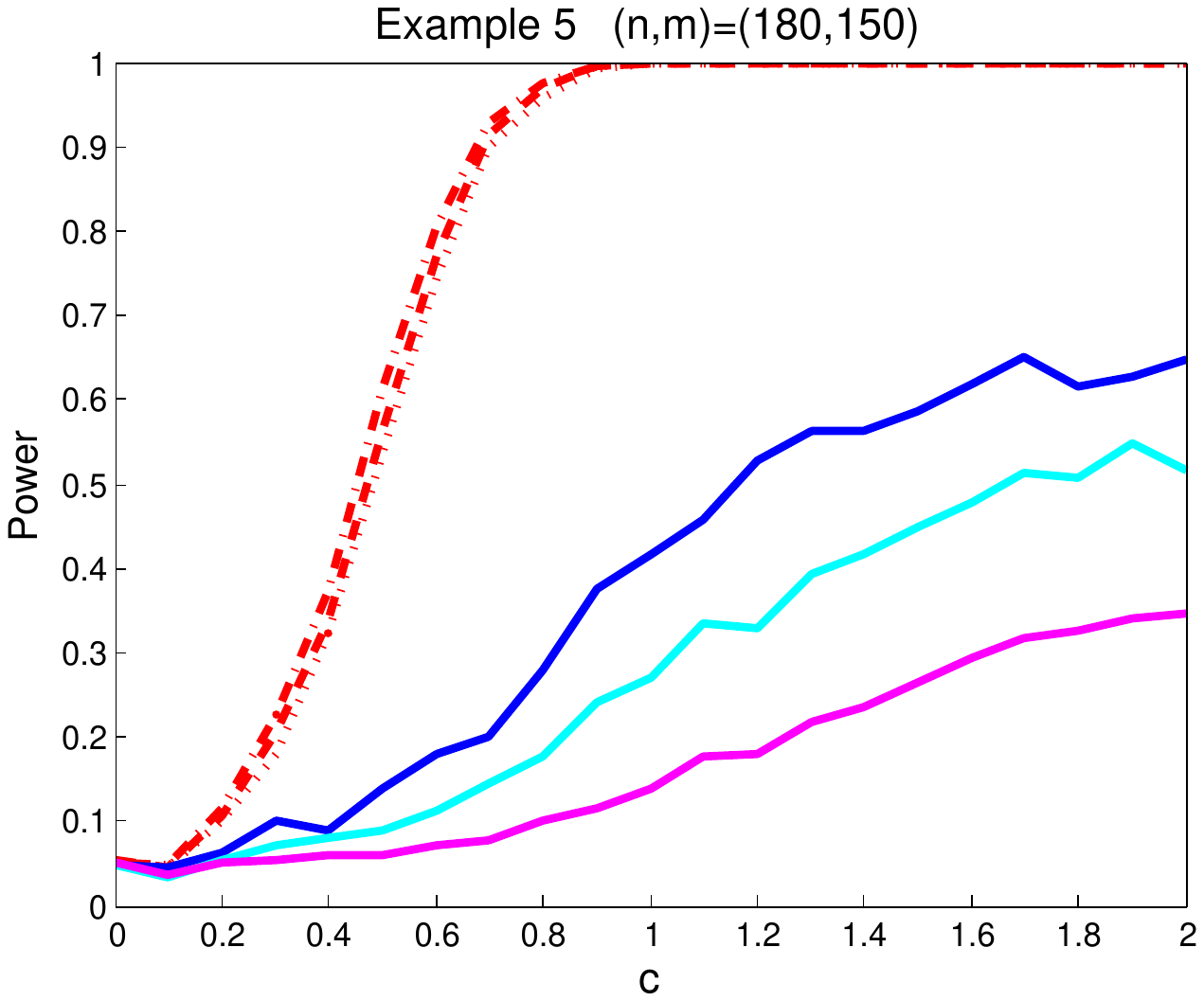}\\
\end{tabular}
\end{figure}

\subsection{Multivariate case}

The computation of the proposed multivariate smooth test and the critical value requires to find optimal directions $\hat{u}_{\max}$ and $\hat{u}_{\max}^{{\rm MB}}$ on the unit sphere $\mathcal{S}^{p-1}$ that maximize non-smooth objective functions (\ref{multivariate.test.statistic}) and (\ref{bootstrap.statistic}), respectively. To solve these optimization problems, we convert the data into spherical coordinates and employ the Nelder-Mead algorithm. As a trade-off between the power and the computational feasibility of the test, we keep the value of $d$ fixed at 4.

Similar to the univariate case, we first carry out 5000 simulations with nominal significance level $\alpha=0.05$ to calculate the empirical sizes of the proposed test T-$\Phi_\alpha^{{\rm MS}}(d)$ with trigonometric series. For each $p \in \{3,5,10\}$, the data are generated from multivariate normal and $t$-distributions with different degrees of freedom ($4$ and $8$) and covariance structures ($I_p$ and $\Sigma$). Sample sizes $(n,m)$ are taken to be (180,160). We summarize the results in Table \ref{multi.size.table}, comparing with the method proposed by \cite{Baringhaus_Franz_2004}, which will be referred as the BF test. From Table~\ref{multi.size.table} we see that when $p=3,5$, both methods have an empirical size fairly close to 0.05; when $p=10$, the empirical size of the proposed smooth test increases since the optimization over the unit sphere becomes more challenging, while the empirical size of the BF test is typically smaller than the nominal level.

\begin{table}[h]
\centering
\caption{\label{size}Empirical size with significance level $\alpha=0.05$. }
\label{multi.size.table}
\begin{tabular}{lcccccccc}
\toprule
&&\multicolumn{2}{c}{Multivariate Normal}&&
\multicolumn{4}{c}{Multivariate t}\\
\cline{3-4}\cline{6-9}\vspace{-0.2cm}\\
&&$N(0,I_p)$ &$N(0,\Sigma)$&&$t_4(0,I_p)$&$t_8(0,I_p)$&$t_4(0,\Sigma)$&$t_8(0,\Sigma)$\\
 \midrule
\multirow{ 2}{*}{$p =3$}&$\text{T-}{\Phi}_{\alpha}^{\text{MS}}(d)$&$0.0446$ &$0.0456$&~&$0.0514$&$0.0442$&$0.0494$ &$0.0458$\\
&BF&$0.0480$ &$0.0494$&~&$0.0504$&$0.0488$&$0.0448$ &$0.0484$\\
\midrule
\multirow{ 2}{*}{$p =5$}&$\text{T-}{\Phi}_{\alpha}^{\text{MS}}(d)$&$0.0496$&$0.0472$&~&$0.0494$&$0.0560$&$0.0450$&$0.0514$\\
&BF&$0.0466$&$0.0472$&&$0.0502$&$0.0458$&$0.0488$&$0.0484$\\
\midrule
\multirow{ 2}{*}{$p =10$}&$\text{T-}{\Phi}_{\alpha}^{\text{MS}}(d)$&$0.0582$&$0.0594$&~&$0.0512$&$0.0516$&$0.0570$&$0.0602$\\
&BF&$0.0422$&$0.0454$&&$0.0364$&$0.0422$&$0.0482$&$0.0438$\\
\bottomrule
\end{tabular}
\end{table}

The power performance of the multivariate smooth test is evaluated through Examples \ref{eg6}--\ref{eg9}. The first two are multivariate versions of Examples \ref{eg1} and \ref{eg4}, which demonstrate, respectively, the alternations with local feature and high frequency. The last two examples are designed to examine a rotation effect in the alternations.  In each one, the power reported is based on 1000 simulations where samples sizes $(n,m)$ are taken to be $(180,160)$. Again, we compare the power of the trigonometric series based smooth test T-$\Phi_\alpha^{{\rm S}}(d)$ with that of the BF test. The power curve are depicted in Figure~\ref{multi.Power.compare}.

\begin{Example}\label{eg6}
\begin{align*}
&X=(X_1,X_2,X_3)^{\intercal}, \ \  X_1,X_2 \overset{iid}{\sim} \text{uniform}\,(-1,1), \ \  X_3=0.3X_1+0.7X_2 \ \ \mbox{versus} \\
&Y=(Y_1, Y_2, Y_3)^{\intercal}, \ \ Y_1, Y_2\overset{iid}{\sim} g_\mu(x) =\frac{1}{2}+2x \frac{\mu-|x|}{\mu^2}  I(|x|<\mu)  \, (0 \le \mu\le 1)\\
&Y_3=0.3Y_1+0.7Y_2.
\end{align*}
\end{Example}

\begin{Example}\label{eg7}
\begin{align*}
&X=(X_1,X_2,X_3)^{\intercal}, \ \ X_1,X_2 \overset{iid}{\sim} \text{uniform}\,(0,1), \ \  X_3=0.3X_1+0.7X_2  \ \  \mbox{versus} \\
&Y=(Y_1, Y_2, Y_3)^{\intercal}, \ \ Y_1, Y_2\overset{iid}{\sim} g_c(x) =\exp\{ c\sin(5\pi x)\} \, (0\le c \le 2) , \, Y_3=0.3Y_1+0.7Y_2.
\end{align*}
\end{Example}

\begin{Example}\label{eg8}
\begin{align*}
& X\sim \mathcal{N}(0, I_5) \ \ \mbox{versus}  \ \  Y=AZ, \  \ Z\sim \mathcal{N}(0, I_5) ,  \\
&  \mbox{ where } A=\left(\begin{array}{cc}
A_0& 0\\
0& I_{3}
\end{array}\right),  \  \   A_0 = \left(\begin{array}{cc}
\sqrt{1-\delta}& \sqrt{\delta}\\
\sqrt{\delta}& \sqrt{1-\delta}
\end{array}\right) \,   (0\le \delta \le 0.5).
\end{align*}
\end{Example}

\begin{Example}\label{eg9}
\begin{align*}
& X\sim t_4(0, I_5) \ \ \mbox{versus}  \  \  Y=AZ, \quad Z\sim t_4(0, I_5) ,   \\
& \mbox{ where }
 A=\left(\begin{array}{cc}
A_0& 0\\
0& I_{3}
\end{array}\right), \  \ A_0 = \left(\begin{array}{cc}
\sqrt{1-\delta}& \sqrt{\delta}\\
\sqrt{\delta}& \sqrt{1-\delta}
\end{array}\right) \,  (0\le \delta \le 0.5).
\end{align*}
\end{Example}

\begin{figure}[!htbp]
\centering
\caption{Empirical powers for Examples~\ref{eg6}--\ref{eg9} based on 1000 replications with $\alpha=0.05$}
\label{multi.Power.compare}
\begin{tabular}{cc}
\includegraphics[trim = 35mm 85mm 40mm 80mm, clip,width=2.8in,height=4.5cm]{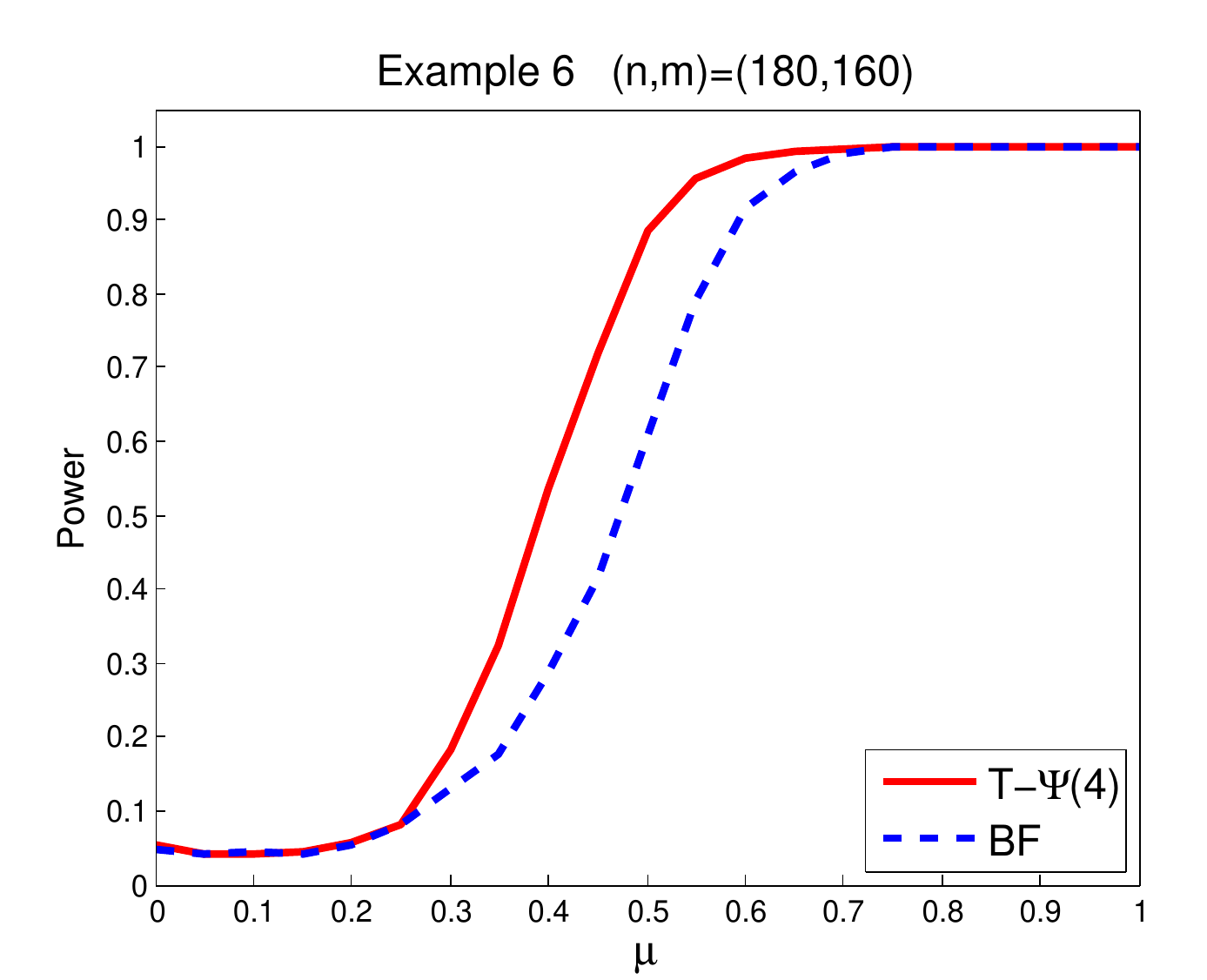}
&\includegraphics[trim = 40mm 85mm 35mm 80mm, clip,width=2.8in,height=4.5cm]{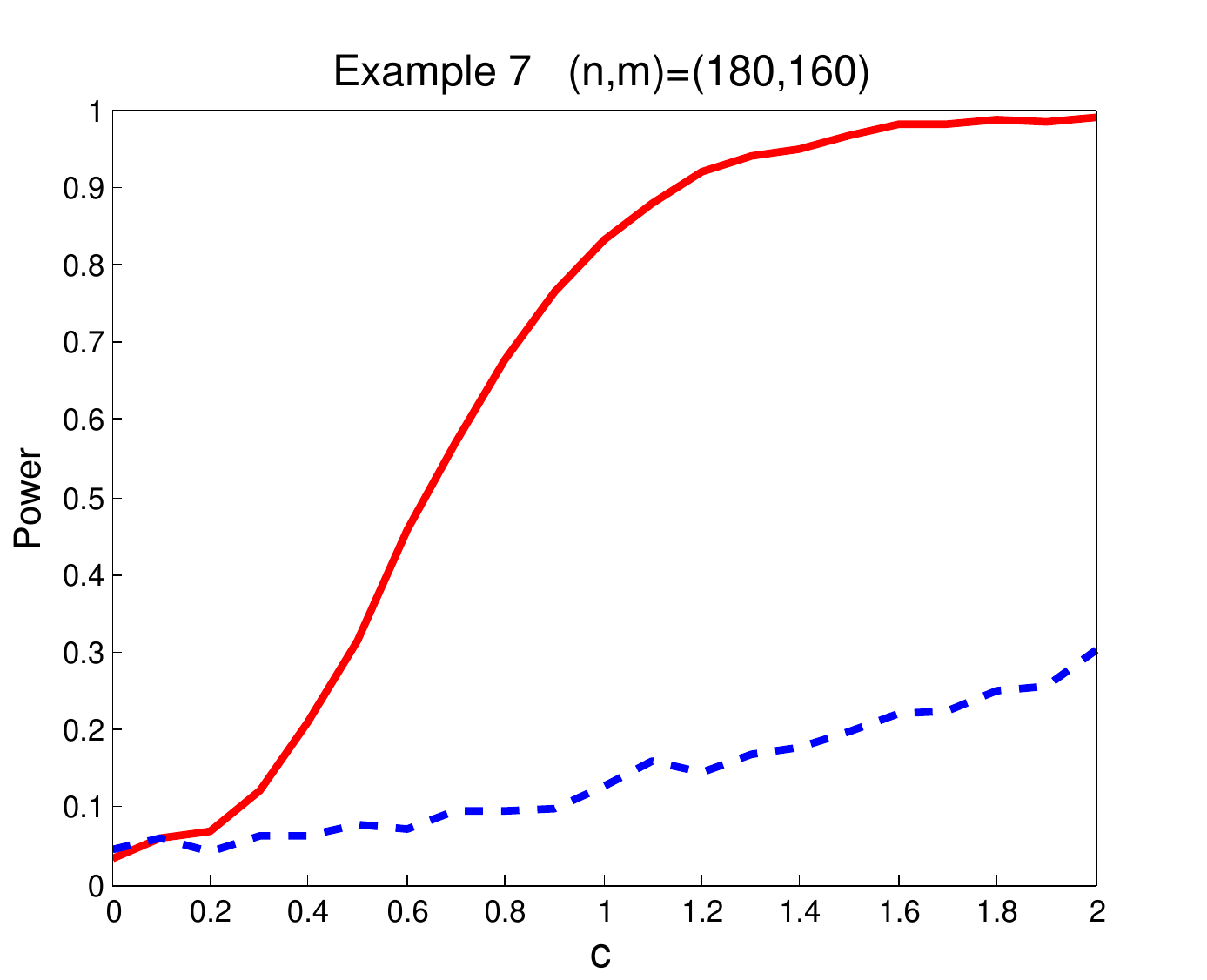}\\
\includegraphics[trim = 35mm 85mm 40mm 80mm, clip,width=2.8in,height=4.5cm]{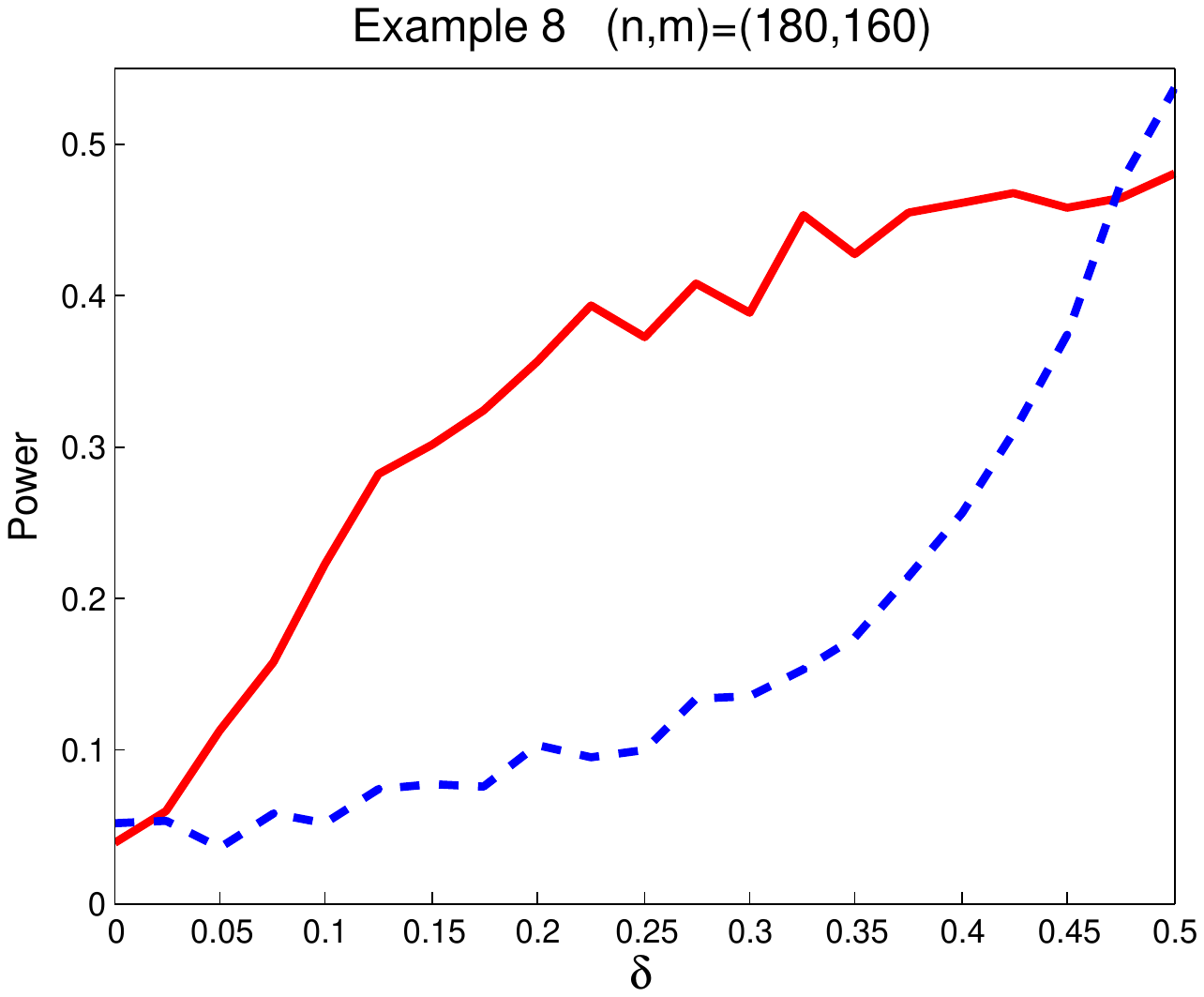}
&\includegraphics[trim = 40mm 85mm 35mm 80mm, clip,width=2.8in,height=4.5cm]{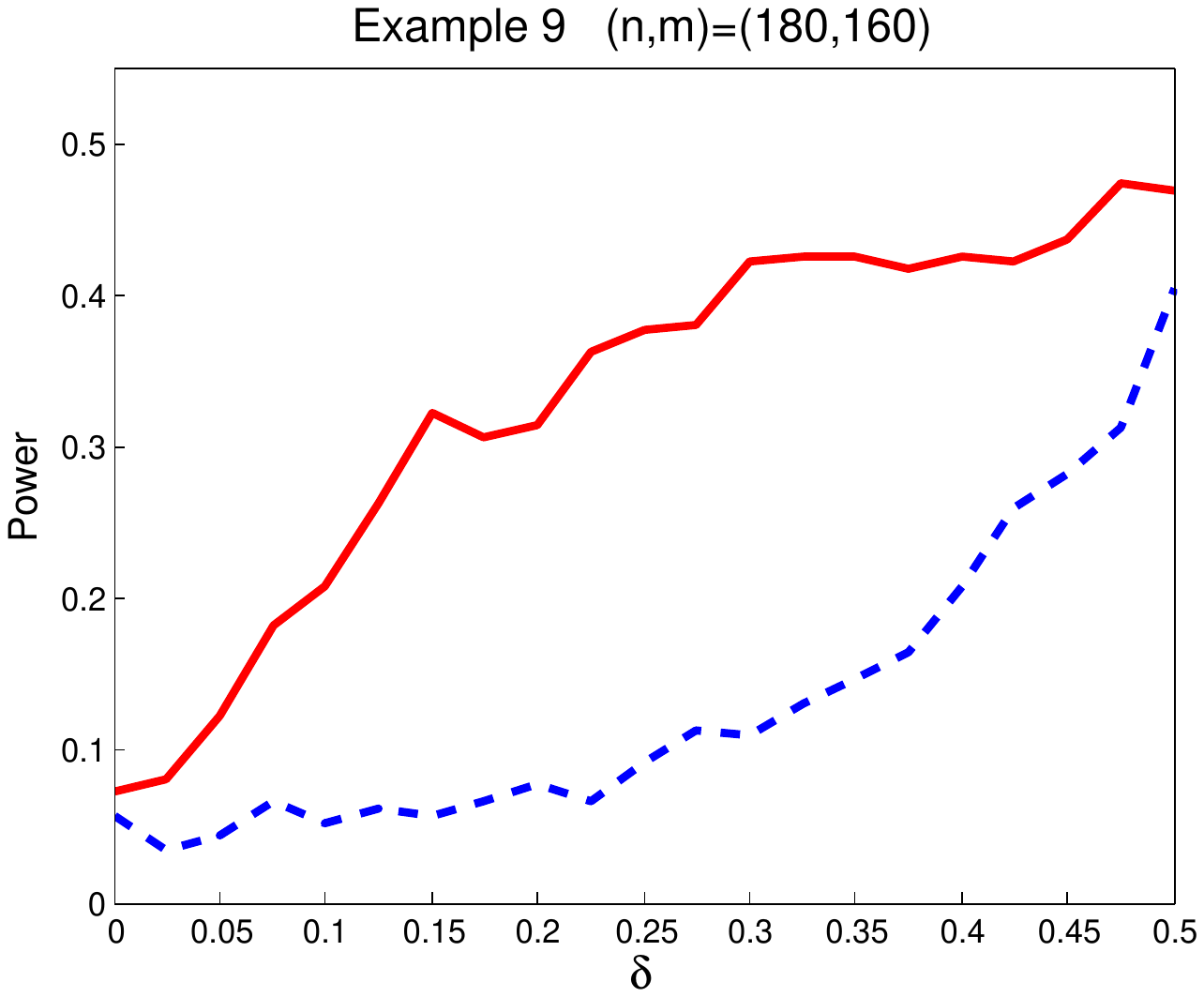}\\
\end{tabular}
\end{figure}

Figure \ref{multi.Power.compare} shows that the proposed smooth test uniformly outperforms the BF test in all the examples in terms of power. Since we are using trigonometric series, the test is powerful especially if the data contains high frequency components (Example \ref{eg7}), which is difficult to be detected by the BF test.

\section{Discussion}

We introduced in this paper a smooth test for the equality of two unknown distributions, which is shown to maintain the pre-specified significance level asymptotically. Moreover, it was shown theoretically and numerically that the test is especially powerful in detecting local features or high-frequency components.


The proposed procedure depends on a user-specific parameter $d$, which is the number of orthogonal directions used to construct the test statistic. Theoretically, the size of $d$ is allowed to grow with $n$ and can be as large as $o(n^c)$ for some $0<c<1$. Since the optimal value of $d$ depends on how far the two unknown distributions deviate from each other, it is not possible to practically define an optimal choice of $d$. As suggested by our numerical studies, $d=10$ is a reasonable choice when the sample sizes are in the order of $10^2$, which leads a good compromise between the computational cost and the performance of the test. Alternatively, a data-driven approach based on a modification of Schwarz's rule was proposed by  \cite{inglot1997data}, that is, $\hat{d} =\argmax_{1\le d \le D(n,m)}\{T(d) - d \log(n+m)\}$ for some $D(n,m)\rightarrow \infty$ as $\min(n,m) \to \infty$, where $T(d)$ is the test statistic using the first $d$ orthonormal functions. This principal can be applied to the proposed testing procedure by setting $D(n,m)$ to be some large value, say $20$. Nevertheless, the optimal choice of $D(n,m)$ remains unclear.


The computation of the multivariate test statistic $\hat{\Psi}_{\max}(d)$ requires solving the optimization problem with an $\ell_2$-norm constraint. To solve this problem when the dimension $p$ is relatively small, we first convert the data into spherical coordinates and then use the Nelder-Mead algorithm. An interesting extension is to combine our method with the smoothing technique as in \cite{H92}. Let $K:\br \mapsto \br$ be a symmetric, bounded density function. For a predetermined small number $h=h_n>0$, $\hat{\psi}_{u,k}$ is approximated by a continuous function $\hat{\psi}_{u,k,h} = m^{-1}\sm  \psi_k(\hat{V}^u_{j,h})$, where
\be
	\hat{V}^u_{j,h}= \frac{1}{n}\sn \mathcal{K} \bigg\{  \frac{ u^{\intercal}( Y_j- X_i)}{h } \bigg\} \ \ \mbox{ with } \ \ \mathcal{K}(t)=\int_{-\infty}^t K(z)\, dz.   \nn 
\ee
As $h  \rightarrow 0$, $\hat{V}^u_{j,h}$ converges to $\hat{V}^u_j$ almost surely, and hence for each $k\in [d]$, $ \sup_{u\in \mathcal{S}^{p-1}}| \hat{\psi}_{k,u,h}|$ is a smoothed version of $ \sup_{u\in \mathcal{S}^{p-1}}| \hat{\psi}_{u,k}|$. The smoothing technique can be similarly applied to the multiplier bootstrap statistic. Consequently, we can employ the gradient descent algorithm to solve the optimization for smooth functions. We leave a thorough comparison of various algorithms for different values of $p$ as an interesting problem for future research.

\section{Proof of the main results}
\label{proof.section}

In this section we prove Theorems~\ref{asymptotic.size}--\ref{multivariate.asymptotic.size}. Proofs of the lemmas and some additional technical arguments are given in the Appendix. Throughout this section, we write $N=n+m$ and use $C$ and $c$ to denote absolute positive constants, which may take different values at each occurrence.  We write $a \lesssim b$ if $a$ is smaller than or equal to $b$ up to an absolute positive constant, and $a \gtrsim b$ if $b\lesssim a$.

\subsection{Proof of Theorem~\ref{asymptotic.size}}

Recall that $\mathbf{G}=(G_1,\ldots, G_d)^{\intercal}$ is a $d$-dimensional standard Gaussian random vector, the distribution of $|\mathbf{G}|_\infty$ is absolute continuous so that $P\{ |\mathbf{G}|_\infty \geq  c_\alpha(d) \}=\alpha$. Therefore, under the assumption that $d\leq n \wedge m$, the conclusion \eqref{size.consist} follows from the following proposition immediately.

\medskip
\begin{proposition} \label{thm.limiting.dist}
Assume that the conditions of Theorem~\ref{asymptotic.size} hold and let
\be
	\gamma_{0n} =\frac{ (\log n)^{7/8}}{n^{1/8}} B_{0d}   , \ \ \gamma_{1n}  =\frac{ (\log d)^{3/2}}{\s{n}}  B_{1d} , \ \ \gamma_{2n}  = \s{\frac{\log d}{n}}  B_{2d} . \label{c123n}
\ee
Then under $H_0:F=G$,
\begin{align}
	   \sup_{t \geq 0}  \big| P\big\{  \hat \Psi(d) \leq t \big\}  - P\big( |\mathbf{G}|_\infty   \leq  t \big) \big| \lesssim    \gamma_{0n}  + \s{\gamma_{1n}} + \s{\gamma_{2n } } . \label{gassian.approxi.1}
\end{align}
\end{proposition}

The proof of Proposition~\ref{thm.limiting.dist} is provided in Section~\ref{proof.proposition.1}. \qed

\subsection{Proof of Theorem~\ref{asymptotic.power.1}}
\label{power.proof.sec}

For the $d$-dimensional Gaussian random vector $\mathbf{G}$, applying the Borell-TIS (Borell-Tsirelson-Ibragimov-Sudakov) inequality \citep{van der Vaart_Wellner_1996} yields that for every $t>0$, $P\{ |\mathbf{G}|_\infty > E( |\mathbf{G}|_\infty ) + t \} \leq \exp(-t^2/2)$. By taking $t=\sqrt{2\log(1/\alpha)}$, we get
\be
	c_\alpha(d ) \leq E ( |\mathbf{G}|_\infty  ) +\s{2\log(1/\alpha)},  \label{cv.ubd}
\ee
where $c_\alpha(d )$ denotes the $(1-\alpha)$-quantile of $|\mathbf{G} |_\infty$. A standard result on Gaussian maxima yields $E( |\mathbf{G}|_\infty) \leq \{1+(2\log d)^{-1}\} \s{2\log d}$.

Let $k^* = \arg\max_{k\in[d]}|\theta_k|$ under $H^d_{1}$ and assume without loss of generality that $\theta_{k^*}>0$. By \eqref{dec.1} and \eqref{Hoeffding.dec} in the proof of Proposition~\ref{thm.limiting.dist}, we have
\begin{align}
	 P_{H^d_{1}} \big\{ \hat{\Psi}  > c_\alpha(d)  \big\} & \geq P_{H^d_{1}} \bigg\{ \s{\frac{nm}{N}} \,\hat{\psi}_{k^*} > c_\alpha(d)  \bigg\}  \nn \\
	& = P_{H^d_{1}} \bigg\{ \frac{1}{\s{N} } \sum_{j=1}^N \xi_{j k^*}   + \s{\frac{nm}{N}} (R_{1k^*}+R_{2k^*}) > c_\alpha(d)  \bigg\},  \label{lower.bound.1}
\end{align}
where $\xi_{jk}= \s{n/m} \, \{ \psi_k(V_j) - \vartheta_k \} I\{j\in [m]\}+ \s{ m/n } \, h_{1k}( X_{j-m} ) I\{j\in m+[n]\}$ for $(j,k)\in [N]\times [d]$ with $V_j=F(Y_j)$, $\vartheta_k=E_{H^d_{1}}\{ \psi_k(V)  \}$ and $h_{1k}(x)=E_{H^d_{1}}(\psi'_k(V)[I\{ V \geq F(x)\}-V])$. Note that $E\{h_{1k}(X)\}=0$ and thus $E(\xi_{jk})=0$. Let $\mathcal{E}(t_1, t_2)$ be as in \eqref{event} for $t_1, t_2 >0$ to be specified. Put $\de= t_1 B_{2d}  + t_2  B_{1d} +\s{nm/N}(\theta_{k^*}-  \vartheta_{k^*})+\s{2\log(1/\alpha)}$, then it follows from \eqref{cv.ubd} and \eqref{lower.bound.1} that
\begin{align}
	& P_{H^d_{1}} \big\{ \hat{\Psi}  > c_\alpha(d)  \big\} \nn  \\
	& \geq P_{H^d_{1}}  \bigg\{ \frac{1}{\s{N} } \sum_{j=1}^N \xi_{j k^*}  > \bigg( 1+ \frac{1}{2\log d} \bigg)\s{2\log d} + \de - \s{\frac{nm}{N}}\, \theta_{k^*} \bigg\} - P\big\{ \mathcal{E}(t_1, t_2)^{{\rm c}} \big\}  \nn \\
	& \geq P_{H^d_{1}}  \bigg\{ \frac{1}{\s{N} } \sum_{j=1}^N \xi_{j k^*}  > \bigg( \frac{1}{2\log d} - \frac{\varepsilon}{2} \bigg)\s{2\log d} + \de \bigg\} - P\big\{ \mathcal{E}(t_1, t_2)^{{\rm c}} \big\}. \nn
\end{align}
In particular, taking $t_1 =t_{1n}(d) \asymp  n^{-1/2} \s{ \log d }$ and $t_2 =t_{2n}(d) \asymp n^{-1/2} \log d$ implies by \eqref{upper.est.3} that $P\{ \mathcal{E}(t_1, t_2)^{{\rm c}} \} \rightarrow 0$ as $d \rightarrow \infty$. Further, by \eqref{alternative.mean.1} and the conditions of the theorem, we have $\de = o(\s{\log d} )$. Consequently, as $d\rightarrow \infty$,
\begin{align}
  P_{H^d_{1}} \big\{ \hat{\Psi}  > c_\alpha(d)  \big\}   \geq 1- P_{H^d_{1}} \bigg(  \frac{1}{\s{N} } \sum_{j=1}^N \xi_{j k^*}  \leq  -\frac{\varepsilon}{2}\s{\log d} \bigg) - P\big\{ \mathcal{E}(t_1, t_2)^{{\rm c}} \big\} \rightarrow 1. \nn
\end{align}
This completes the proof of Theorem~\ref{asymptotic.power.1}. \qed

\subsection{Proof of Theorem~\ref{multivariate.asymptotic.size}}
\label{proof.multivariate.size}

We first introduce two propositions describing the limiting null properties of the multivariate smooth and multiplier bootstrap statistics used to construct the test. The conclusion of Theorem~\ref{multivariate.asymptotic.size} follows immediately.

The first proposition characterizes the non-asymptotic behavior of the multivariate smooth statistic $\hat{\Psi}_{\max}(d)$ which involves the supremum of a centered Gaussian process. Let ${\mathcal{F}}={\mathcal{F}}^p_d$ be as in \eqref{heuristic.approximation} and for simplicity, the dependence of ${\mathcal{F}}$ on $(p,d)$ will be assumed without displaying.

\begin{proposition} \label{limiting.null.supremum.statistic}
Suppose that Assumptions~\ref{basic.assumption} and \ref{regularity.assumption.2} hold. Then there exists a centered, tight Gaussian process $\bG$ indexed by ${\mathcal{F}}$ such that under the null hypothesis $H_0:F=G$,
\begin{align}
	\sup_{t\geq 0} \big|  P \big\{ \hat{\Psi}_{\max}(d) \leq t \big\} - P\big(  \| \bG \|_{ {\mathcal{F}}} \leq t \big) \big| \leq C \, n^{-c},
	\label{multivariate.berry-esseen.bound}
\end{align}
where $C$ and $c $ are positive constants depending only on $c_0, c_1, C_0$ and $C_1$.
\end{proposition}

Proposition~\ref{limiting.null.supremum.statistic} implies that the ``limiting'' distribution of $\hat{\Psi}_{\max}$ depends on unknown the covariance structure given in \eqref{sigma.ukvl}. To compute a critical value we suggest to use multiplier bootstrapping as described in Section~\ref{multivariate.cv.sec}. The following result, which can be regarded as a multiplier central limit theorem, provides the theoretical justification of its validity. In fact, the construction of the multiplier bootstrap statistic $\hat{\Psi}_{\max}^{{\rm MB}}(d)$ involves the use of artificial random numbers to simulate a process, the supremum of which is (asymptotically) equally distributed as $ \| \bG  \|_{ {\mathcal{F}} }$ according to Proposition~\ref{multiplier.bootstrap.theorem} below.

\begin{proposition} \label{multiplier.bootstrap.theorem}
Suppose that Assumptions~\ref{basic.assumption} and \ref{regularity.assumption.2} hold. Then with probability at least $1-3n^{-1}$,
\begin{align}
	\sup_{t\geq 0} \big|  P_e \big\{  \hat{\Psi}_{\max}^{{\rm MB}}(d) \leq t \big\} - P\big(  \| \bG  \|_{ {\mathcal{F}}} \leq t \big) \big| \leq C  \, n^{-c}
	\label{bootstrap.BE.bound}
\end{align}
for $\bG$ as defined in Proposition~\ref{limiting.null.supremum.statistic}, where $C$ and $c $ are positive constants depending only on $c_0, c_1, C_0$ and $C_1$.
\end{proposition}

Proofs of the above two propositions are given in Section~\ref{proposition.proof}. \qed

\subsection{Proof of Propositions~\ref{thm.limiting.dist}--\ref{multiplier.bootstrap.theorem}}
\label{proposition.proof}

\subsubsection{Proof of Proposition~\ref{thm.limiting.dist}}
\label{proof.proposition.1}

For every $k\in [d]$, it follows from \eqref{generated.Z} and Taylor  expansion that
\begin{align}
	\frac{1}{m}\sm \psi_k(\hat V_j) & = \frac{1}{m}\sm \psi_k(V_j) + \frac{1}{m}\sm \psi'_k(V_j)(\hat V_j - V_j) + \frac{1}{2m}\sm \psi''_k(\xi_j)(\hat V_j - V_j)^2 \nn \\
	& = \frac{1}{m}\sm \psi_k(V_j) + \frac{1}{nm}\sn \sm \psi'_k(V_j)\{ I(X_i\leq Y_j) - F(Y_j) \} +   R_{1k},  \label{dec.1}
\end{align}
where $R_{1k} := (2m)^{-1}\sm \psi''_k(\zeta_j)(\hat V_j - V_j)^2$ and $\zeta_j$ is a random variable lying between $\hat V_j$ and $V_j$. It is straightforward to see that $R_{1k}\leq \frac{1}{2} \|\psi''_k\|_\infty \max_{1\leq j\leq m}(\hat V_j-V_j)^2 $. A direct consequence of the Dvoretzky-Kiefer-Wolfwitz inequality \citep{Massart_1990}, i.e. for every $t>0$, $P\{ \s{n}\sup_{x } \big|F_n(x) -F(x) \big|  > t \} \leq 2 \exp(-2 t^2)$, is that
\begin{align}
	P \bigg(   n \max_{1\leq k\leq d} | R_{1k} | \big/ \|\psi''_k \|_\infty >  t \bigg) \leq 2 \exp(-4t). \label{R1k.tail.prob}
\end{align}

Let $h_k(x, y)  =  \psi_k'(F(y)) \{ I(x\leq y) - F(y) \}$ for $x, y \in \br$ be a kernel function $\br \times \br \mapsto \br$. Then the second addend on the right side of \eqref{dec.1} can be written as $U_{n,m}(k) =(nm)^{-1}\sn \sm h_k(X_i, Y_j)$ with $E  \{ h_k(X, Y)\}  =0$. Observer that $U_{n,m}(k)$ is a two-sample $U$-statistic with a bounded kernel $h_k $ satisfying $b_k:= \|h_k \|_\infty \leq \| \psi_k'\|_\infty$ and
\be
 \sigma_k^2:= E \{ h_k(X,Y)^2 \}  =E  \{ (V-V^2) \psi_k'(V)^2 \} \leq  \| \psi_k'\|_\infty^2/4.  \label{kernel.property}
\ee
Let $h_{1k}(x)=E_{H_0^d} \{ h_k(X, Y) | X=x  \} $ and $h_{2k}(y)=E_{H_0^d} \{ h_k(X, Y) | Y=y  \}$ be the first order projections of the kernel $h_k$ under $H_0^d$. Since $X$ and $Y$ are independent and under $H_0$, $V = F(Y) =_d {\rm Unif}(0,1)$ under $H_0$, we have $h_{2k}  \equiv 0$ and
$$
	h_{1k}(x) =  E \big( \psi_k'(V)  [ I \{ V \geq F(x) \} - V ] \big) = \int_{F(x)}^1 \psi'_k(v) \, d v - \int_0^1 v\psi'_k(v)\, dv =-\psi_k(F(x)).
$$
Define random variables $U_i=F(X_i) =_d {\rm Unif}(0,1)$ that are independent of $V_j=F(Y_j)$. Then, using the Hoeffding's decomposition gives
\begin{align}
	U_{n,m}(k) = -\frac{1}{n} \sn \psi_k(U_i) + \frac{1}{nm}\sn \sm h_{0k}(X_i, Y_j) :=- \frac{1}{n}\sn \psi_k(U_i) + R_{2k}, \label{Hoeffding.dec}
\end{align}
where $h_{0k}(x,y) =h_k(x,y)-h_{1k}(x)-h_{2k}(y)$.

In view of \eqref{dec.1} and \eqref{Hoeffding.dec}, we introduce a new sequence of independent random vectors $\{\bxi_j=(\xi_{j1},\ldots,\xi_{jK})^{\intercal}\}_{j=1}^{N}$ for $N=n+m$, defined by
\be
	\xi_{jk} =  \left\{
           \begin{array}{ll}
      \s{n/m} \, \psi_k(V_j)  \quad  &   1\leq j \leq m,  \\
     -\s{m/n}\, \psi_k( U_{j-m} )  \quad  &  m+1 \leq j \leq N.             \end{array}   \right. \label{xi.def}
\ee
Put $\bpsi=(\psi_1,\ldots,\psi_d)^{\intercal}$, $\mathbf{R}_1=(R_{11},\ldots, R_{1d})^{\intercal}$ and $\mathbf{R}_2=(R_{21},\ldots, R_{2d})^{\intercal}$, such that
$$
	\s{\frac{n}{mN}}\sm \bpsi(\hat V_j ) = \frac{1}{\s{N}}\sum_{j=1}^{N} \bxi_{j} + \s{\frac{nm}{N}}  ( \mathbf{R}_1 + \mathbf{R}_2 ).
$$
Recall that $\{\psi_0 \equiv 1, \psi_1, \ldots, \psi_d\}$ is a set of orthonormal functions and $V =_d  {\rm Unif}(0,1)$ under $H_0$. By \eqref{xi.def}, the covariance matrix of $N^{-1/2}\sum_{j=1}^{N} \bxi_{j}$ is equal to $ \mathbf{I}_d$.

For any $t_1, t_2>0$, define the event
\be
	\mathcal{E}(t_1, t_2) =  \bigcap_{k=1}^d \big\{  \s{m}|R_{1 k}|  \leq  \| \psi''_k \|_\infty t_1 \big\} \cap \big\{ \s{m}|R_{2 k}|   \leq  \| \psi'_k \|_\infty t_2   \big\}.  \label{event}
\ee
Under $H_0$, we have for every $t>0$,
\begin{align}
 	&P_{H_0}\big\{   \hat \Psi(d)   \leq t \big\}  \nn \\
 	& = P\bigg\{  \max_{1\leq k\leq d} \bigg|  \s{\frac{n}{mN}} \sm \psi_k(\hat V_j) \bigg| \leq t \bigg\} \nn \\
 	& =   P\bigg\{  \max_{1\leq k\leq d} \bigg| \frac{1}{\s{N}} \sum_{j=1}^{N} \xi_{jk} + \s{\frac{nm}{N}}(R_{1k} + R_{2k}) \bigg| \leq t \bigg\} \nn \\
 	& \leq P\bigg\{  \max_{1\leq k\leq d} \bigg| \frac{1}{\s{N}} \sum_{j=1}^{N} \xi_{jk}  \bigg| \leq t + \s{\frac{n}{N}} \big(  t_1 B_{2d}  + t_2  B_{1d}  \big) \bigg\} + P\big\{ \mathcal{E}(t_1, t_2)^{{\rm c}} \big\} ,  \label{upper.est.1}
\end{align}
where $B_{\ell d}$ ($\ell=1,2$) are as in \eqref{Bd.def}. To get rid of the absolute value in \eqref{upper.est.1}, a similar argument as in the proof of Theorem~1 in \cite{chang2014simulation} gives
\beqn
 P\bigg(  \max_{1\leq k\leq d} \bigg| \frac{1}{\s{N}}\sum_{j=1}^{N}\xi_{jk}  \bigg| \leq t \bigg)  = P \bigg(   \max_{1\leq k\leq 2d }  \frac{1}{\s{N}} \sum_{j=1}^{N} \xi^{{\rm ext}}_{jk} \leq t  \bigg),  \label{prob.dec}
\eeqn
where $\{\bxi^{{\rm ext}}_j\}_{j=1}^N$ is a sequence of dilated random vectors taking values in $\br^{2d}$ defined by $\bxi^{{\rm ext}}_j=(\xi^{{\rm ext}}_{j 1}, \ldots, \xi^{{\rm ext}}_{j,2d})^{\intercal}=(\bxi^{\intercal}_j, - \bxi^{\intercal}_j)^{\intercal}$. In view of \eqref{prob.dec}, we only need to focus on $ \max_{1\leq k\leq d} N^{-1/2} \sum_{j=1}^{N} \xi_{jk}$ without losing generality.

Note that $\xi_{jk}$ are bounded random variables satisfying $E (\xi_{jk})=0$ and $|\xi_{jk}| \leq \s{\frac{n}{m}} \|\psi_k\|_\infty$. Applying Lemma~2.3 and Lemma~2.1 in \cite{chernozhukov2013gaussian}, respectively, yields
\begin{align}
	& \sup_{t\in \br} \bigg| P\bigg( \max_{1\leq k\leq d} \frac{1}{\s{N}} \sum_{j=1}^{N} \xi_{jk} \leq t	\bigg) - P \bigg( \max_{1\leq k\leq d} G_k  \leq t  \bigg)  \bigg| \lesssim     \frac{ \{\log(dn)\}^{7/8} }{n^{1/8} }B_d,  \nn
\end{align}
where $B_d : =  [ E \{ \max_{1\leq k\leq d} |\psi_k(V)|^3 \}  ]^{1/4} \leq B_{0d}^{3/4}$, $\mathbf{G}=(G_1,\ldots, G_d)^{\intercal}  =_d N(0, \mathbf{I}_d)$ and for every $\varepsilon>0$,
\begin{align}
 \sup_{t\in \br} P\bigg( \bigg| \max_{1\leq k\leq d} G_k - t \, \bigg| \leq  \varepsilon \bigg) \leq  4 \varepsilon \big(1 +\sqrt{2\log d} \,\big). \nn
\end{align}
The last two displays jointly imply
\begin{align}
 & P \bigg\{  \max_{1\leq k\leq d}  \frac{1}{\s{N}} \sum_{j=1}^{N} \xi_{jk}  \leq t + \s{\frac{n}{N}} (  t_1  B_{2d}  +  t_2  B_{1d} ) \bigg\} \nn \\
	& \leq   P \bigg(  \max_{1\leq k\leq d} G_k  \leq t  \bigg) + C \bigg\{   \frac{\{\log(d n)\}^{7/8}}{n^{1/8} } B_{0d}^{3/4} + (\log d)^{1/2} ( t_1  B_{2d}  +  t_2  B_{1d} )  \bigg\} . \label{upper.est.2}
\end{align}

For $P\{ \mathcal{E}(t_1, t_2)^{{\rm c}} \}$ in \eqref{upper.est.1}, it follows from \eqref{R1k.tail.prob} and \eqref{R2k.tail.prob} in Lemma~\ref{lemma.1} that
\begin{align}
 P\big\{ \mathcal{E}(t_1, t_2)^{{\rm c}} \big\} & \leq  2 \exp(-4  t_1 n/\s{m}) + \sum_{k=1}^d P\big( \s{m}|R_{2k}| > \| \psi'_k\|_\infty t_2/2 \big) \nn \\
 &  \lesssim  \exp(-4t_1 \s{n} \, )+  d \exp(-c   t_2 \s{n}   )   .   \label{upper.est.3}
\end{align}
Taking $t_1 \asymp (\gamma_{2n}  n)^{-1/2}$, $t_2 \asymp  ( \gamma_{1n}  n)^{-1/2}\log d$ in \eqref{upper.est.1} implies by \eqref{upper.est.2} and \eqref{upper.est.3} that
\beq
	P_{H_0}\big(  \hat \Psi  \leq t \big)  \leq P \big( |\mathbf{G}|_\infty \leq t \big) + C \bigg[    \frac{ \{\log(dn)\}^{7/8} }{n^{1/8}}  B_{0d}^{3/4}  +  \s{ \gamma_{1n} } + \s{ \gamma_{2n} } \bigg],
\eeq
where $\gamma_{\ell n}$ ($\ell=1,2$) are as in \eqref{c123n}. Here, the last inequality relies on the fact that $\sup_{t \geq 0} ( t e^{-t} ) \leq e^{-1}$. A similar argument leads to the reverse inequality and thus completes the proof.  \qed

\subsubsection{Proof of Proposition~\ref{limiting.null.supremum.statistic}}
\label{proof.multivariate.theorem}

In view of \eqref{constraint.pd}, we assume without loss of generality that $B_{1d} \leq \s{n}$. Let $\mathcal{T}=\mathcal{T}^p_d$ be the product space $\mathcal{S}^{p-1} \times [d]$. For every $u\in \mathcal{S}^{p-1}$, let $\hat{V}^u_j= F_n^u(u^{\intercal} Y_j)$, $ V^u_j = F^u(u^{\intercal} Y_j)$, $j\in [m]$ and $U^u_i=F^u(u^{\intercal} X_i)$, $i\in [n]$. By Taylor expansion and arguments similar to those employed in the proof of Proposition~\ref{thm.limiting.dist}, we obtain that for every $(u,k)\in \mathcal{T}$,
\begin{align}
	 & \s{\frac{nm}{N}} \hat \psi_{u,k}   =\s{\frac{n}{mN}}\sm \psi_k(\hat{V}^u_j)  \nn  \\
	& = \s{\frac{n}{mN}}\sm \big\{ \psi_k(V_j^u)  - \psi'_k(V_j^u) V_j^u  \big\} +\s{\frac{1}{nmN}} \, U_{n,m}(u,k)   +  \s{\frac{nm}{N}} R_{u,k}, \label{decomposition.psi.uk}
\end{align}
where $U_{n,m}(u,k) : =\sn \sm   \psi'_k(V_j^u) I\{u^{\intercal}(X_i  - Y_j)\leq 0 \}$ is a two-sample $U$-statistic with $E\{U_{n,m}(u,k)\}=\psi_k(1)$ under $H_0$ and $| R_{u,k} |\leq \frac{1}{2}\| \psi''_k \|_\infty  \max_{j\in [m]} (\hat{V}^u_j-V^u_j)^2$. Let
\be
	\mathcal{H}=\mathcal{H}^p_d= \big\{ h_{u,k}(\cdot,\cdot) :\br^p \times \br^p \mapsto \br \,\big| (u,k)\in \mathcal{T} \big\} \label{class.H}
\ee
be a class of measurable functions, where $h_{u,k}(x,y) = \psi'_k( F^u(u^{\intercal} y)) I\{u^{\intercal}(x  -y)\leq 0\}$ for $x, y \in \br^p$. For ease of exposition, the dependence of $\mathcal{T}$ and $\mathcal{H}$ on $(p,d)$ will be assumed without displaying. In the above notation, each $h=h_{u,k} \in \mathcal{H}$ determines a two-sample $U$-statistic $U_{n,m}(h):= \sn \sm h(X_i, X_j)  =U_{n,m}(u,k)$, such that $\{U_{n,m}(h)\}_{h\in \mathcal{H}}$ forms a two-sample $U$-process indexed by the class $\mathcal{H}$ of kernels. Moreover, define the degenerate version of $\mathcal{H}$ as
\be
	 {\mathcal{H}}_0 =\big\{ h_0(x,y) =h (x,y)-  (P_Y h) (x) -  ( P_X h) (y) +  (P_X\times P_Y) (h)  \, \big|  h  \in \mathcal{H}  \big\}, \label{degenerate.class.H}
\ee
where for $h(\cdot,\cdot):\br^p \times \br^p \mapsto \br$,
$$
	(P_X h)(\cdot)  = \int_{\br^p} h( x , \cdot) \, dF(x), \quad (P_Y h)(\cdot)   = \int_{\br^p}  h( \cdot, y) \, dG(y)
$$
and $(P_X\times P_Y)(h ) =  \int \int h( x ,y) \, dF(x) \, dG(y) $. Under $H_0$, it is easy to verify that for every $(u,k)\in \mathcal{T}$, $(P_X \times P_Y)(h_{u,k})=\psi_k(1)$,
$$
	(P_Y h_{u,k} )(x) = \psi_k(1) - \psi_k(F^u(u^{\intercal}x)) \ \ \mbox{ and } \ \   (P_X h_{u,k} )(y) = \psi'_k(F^u(u^{\intercal} y))F^u(u^{\intercal} y).
$$

In addition to $\mathcal{H}$ and $\mathcal{H}_0$, we define the following class of measurable functions on $\br^p$:
\begin{align}
	 {\mathcal{F}}   = {\mathcal{F}}^p_d = \big\{  x \mapsto {f}_{u,k}(x) = \psi_k \circ f_u(x)   :  k\in [d], f_u \in \mathcal{F}^p \big\},  \label{class.barF}
\end{align}
where
\begin{align}
	  \mathcal{F}^p = \bigg\{  y \mapsto f_u(y) =  \int I_u(x,y)  \, dF(x)   : I_u \in \mathcal{I}^p \bigg\}  \label{class.FI}
\end{align}
with $\mathcal{I}^p =  \{   (x, y) \mapsto   I\{u^{\intercal}(x-y)\leq 0 \}  : u\in \mathcal{S}^{p-1}  \}$.

Together, \eqref{decomposition.psi.uk} and \eqref{degenerate.class.H}--\eqref{class.FI} lead to
\begin{align}
	\big| \hat{\Psi}_{\max} - \Psi_0 \big| \leq \frac{1}{\s{N}} \bigg\{  \frac{1}{\s{nm}} \| U_{n,m}   \|_{\mathcal{H}_0}  +  \s{nm}\sup_{(u,k)\in \mathcal{T}} | R_{u,k} | \bigg\},   \label{comparison.1}
\end{align}
where $\| U_{n,m} \|_{\mathcal{H}_0}= \sup_{  h_0 \in \mathcal{H}_0} |U_{n,m}( h_0 )|$,
\begin{align}
	\Psi_0 = \sup_{ f \in  {\mathcal{F}} } \bigg| \frac{1}{\s{N}}   \sum_{j=1}^N w_j  \{  {f}(Z_j) - P_X f \} \bigg| \label{leading.term}
\end{align}
with $\{Z_1,\ldots, Z_N\}=\{Y_1,\ldots, Y_m, X_1,\ldots, X_n\}$ and $w_j=\s{n/m} \, I\{j\in [m]\}-\s{m/n} \, I\{j\in [n]+m\}$ for $j=1,\ldots, N$.

With the above preparations, the rest of the proof involves three steps: First, approximation of the test statistic $ \hat{\Psi}_{\max}$ by $\Psi_0$ requires the uniform negligibility of the right side of \eqref{comparison.1}. Second, we prove the Gaussian approximation of $\Psi_0$ by the supremum of a centered, tight Gaussian process $\bG$ indexed by $ {\mathcal{F}}$ with covariance function
\be
	\e \{ ( \mathbb{G} f_{u,k} ) ( \mathbb{G} f_{v, \ell } )  \} = \int_{\br^p} \psi_k(F^u(u^{\intercal} x)) \psi_{\ell}(F^v(v^{\intercal} x)) \, dF(x)  \label{covariance.function}
\ee
for $(u,k), (v,\ell) \in \mathcal{T}$. Finally, we apply an anti-concentration argument due to \cite{chernozhukov2013anti} to construct the Berry-Esseen type bound.

\medskip
\noindent
{\it Step 1.} The following two results show the uniform negligibility of the right side of~\eqref{comparison.1}.

\begin{lemma} \label{negligibility.degenrate.U.process}
Assume that the conditions of Proposition~\ref{limiting.null.supremum.statistic} hold. Then under $H_0:F=G$,
\be
	E \big(  \| U_{n,m}  \|_{\mathcal{H}_0} \big) \lesssim  B_{2d} \s{(p+\log d) n m }.  \label{mean.sup.Unm.ubd}
\ee
\end{lemma}

\begin{lemma} \label{negligibility.EDF}
With probability at least $1-2n^{-1}$, we have
\be
  \sup_{(u,t)\in \mathcal{S}^{p-1}\times \br}  | F^u_n(t) - F^u(t)  | \lesssim \s{\frac{p+ \log n}{n}}. \label{edf.uniform.rate}
\ee
\end{lemma}

By \eqref{mean.sup.Unm.ubd} and \eqref{edf.uniform.rate}, it follows from the Markov inequality that for $t>0$,
\be
	P \big\{ (nm)^{-1/2}  \| U_{n,m}  \|_{\mathcal{H}_0} > t  \big\} \lesssim t^{-1} B_{2d} \s{p+\log d} \label{uniform.negligibility.1}
\ee
for any $t>0$ and with probability at least $1- 2n^{-1}$, $\s{nm}\sup_{(u,k)\in \mathcal{T}} |R_{u,k} | \lesssim B_{2d} \s{p +\log n}$. Taking $t=\ga^{-1} B_{2d}\s{p+\log d}$ for some $\gamma\in (0,1)$ in \eqref{uniform.negligibility.1} implies by \eqref{comparison.1} that
\begin{align}
	P\bigg( \big|   \hat{\Psi}_{\max}  - \Psi_0 \big|  \gtrsim B_{2d}\frac{ \s{p+ \log d + \log n}}{\ga \s{n}}  \,   \bigg) \lesssim  \ga + n^{-1}. \label{comparison.2}
\end{align}

\medskip
\noindent
{\it Step 2.} The following result establishes the Gaussian approximation for $\Psi_0$.

\begin{lemma} \label{Gaussian.approximation.prop}
Assume that the conditions of Proposition~\ref{limiting.null.supremum.statistic} hold. Then under $H_0$, there exists a centered, tight Gaussian process $\mathbb{G}$ indexed by ${\mathcal{F}}=\mathcal{F}^p_d$ given in \eqref{class.barF} with covariance function \eqref{covariance.function} and a random variable $\Psi^* =_d \| \mathbb{G} \|_{ {\mathcal{F}}}=\sup_{f \in  {\mathcal{F}}} |\bG f | $ such that for every $\gamma \in (0,1)$,
\begin{align}
	& P\bigg\{   | \Psi_0 - \Psi^*  |  \gtrsim   B_{1d} \frac{K^p_d \log n}{  \s{ \gamma \, n }}  +   B_{1d}^{1/2} \frac{(K^p_d \log n)^{3/4}}{\ga^{1/2} n^{1/4}}  +   B_{1d}^{1/3} \frac{ (K^p_d \log n)^{2/3} }{\gamma^{1/3} n^{1/6}} \bigg\}  \nn \\
	& \qquad \qquad \qquad \qquad \qquad \qquad \qquad  \qquad \qquad \qquad \qquad  \qquad \qquad   \leq  \gamma + n^{-1} \log n ,  \label{CCK.Gaussian.approximation}
\end{align}
where $K^p_d=p+\log d$.
\end{lemma}

By \eqref{comparison.2} and \eqref{CCK.Gaussian.approximation} with $K^p_d=p+\log d$,
\begin{align}
	P\big\{ \big| \hat{\Psi}_{\max}  - \Psi^* \big| \gtrsim \Delta_{1n}(\ga )   \big\} \lesssim \Delta_{2n}(\ga), \label{comparison.3}
\end{align}
where
\begin{align*}
	 \Delta_{1n}(\ga )  =   B_{2d}  \frac{ ( K^p_d + \log n )^{1/2}}{\ga \s{n}} +  B_{1d}\frac{K^p_d  \log n}{ \gamma^{1/2} \s{n }}  + B_{1d}^{1/2}\frac{ (K^p_d  \log n)^{3/4}}{\ga^{1/2} n^{1/4}} +  B_{1d}^{1/3} \frac{(K^p_d  \log n)^{2/3} }{\gamma^{1/3}n^{1/6}}
\end{align*}
and $\Delta_{2n}(\ga)=\ga + n^{-1} \log n$.

\medskip
\noindent
{\it Step 3.} Now we restrict attention to the Gaussian supremum $\Psi^*$. By Corollary~2.2.8 in \cite{van der Vaart_Wellner_1996} and \eqref{barF.entropy.bound}, we get
\beq
	E \Psi^*    \lesssim  \int_0^2 \s{\sup_Q  \log N( {\mathcal{F}}, L_2(Q), \varepsilon )} \, d\varepsilon  \lesssim   \s{p+\log d}.
\eeq
Combined with Corollary~2.1 in \cite{chernozhukov2013anti}, this implies for every $\varepsilon \geq 0$ that
\be
	\sup_{t\geq 0 }  P \big(  | \Psi^* - t | \leq \varepsilon \big) \lesssim \varepsilon \s{p+\log d}.  \label{anti-concentration.bound}
\ee

Together, \eqref{comparison.3} and \eqref{anti-concentration.bound} yield, for every $t\geq 0$,
\begin{align}
	 	P \big(  \hat{\Psi}_{\max} \leq t \big)   & \leq P \big\{ \Psi^* \leq t+ C \Delta_{1n}(\gamma) \big\} +C \Delta_{2n}(\gamma) \nn \\
	& \leq  P\big( \Psi^* \leq t  \big) + C \big\{  \Delta_{1n}(\gamma) \s{p+\log d}  + \Delta_{2n}(\gamma) \big\}. \nn
\end{align}
A similar argument leads to the reverse inequality. Finally, in view of \eqref{constraint.pd}, taking
\begin{align*}
	 \gamma  = \gamma_n(p,d)  = \max  \bigg\{ B_{2d}^{1/2} \frac{ (p+\log n)^{1/4} }{n^{1/4}} , & \,B_{1d}^{1/4} (\log n)^{1/2} \frac{p^{7/8}}{n^{1/8}} , \\
	 &    B_{1d}^{1/3}(\log n)^{1/2}   \frac{p^{5/6}}{n^{1/6}}, \, B_{1d}^{2/3} (\log n)^{2/3}\frac{p}{n^{1/3}} \bigg\}
\end{align*}
completes the proof under the assumption $d \leq \min\{n, m, \exp(C_0 p) \}$. \qed

\subsubsection{Proof of Proposition~\ref{multiplier.bootstrap.theorem}}
\label{proof.multiplier.bootstrap}

Throughout the proof, $\{e_i\}_{i=1}^n$ is a sequence of i.i.d. standard normal random variables and $P_e$ denotes the probability measure induced by $\{e_i\}_{i=1}^n$ holding $\{X_i\}_{i=1}^n$ fixed. For every $(u,k)\in \mathcal{T}=\mathcal{S}^{p-1}\times [d]$, by Taylor expansion we have
\begin{align}
	 \frac{1}{\s{n}} \sn e_i \psi_k(\hat{U}^u_i)  =  \frac{1}{\s{n}}\sn e_i \psi_k(U^u_i) + \frac{1}{n^{3/2}}\sn \sum_{j=1}^n \bar{h}_{u,k}(\bar{X}_i, \bar{X}_j)  + \hat{R}_{u,k} , \label{marginal.multiplier.dec}
\end{align}
where $\bar{X}_i=(e_i, X_i^{\intercal})^{\intercal} \in \br^{p+1}$,
$$
	\bar{h}_{u,k}( \bar{x}_1,\bar{x}_2) =e_1\, \psi_k' \circ F^u(u^{\intercal} x_1)  \big[  I\{u^{\intercal} (x_2 - x_1) \leq 0 \}-F^u(u^{\intercal} x_1) \big], \ \  \bar{x}_\ell =(e_\ell , x^{\intercal}_\ell)^{\intercal}  \in \br^{p+1}
$$
for $\ell=1 ,2 $ and the remainder $\hat{R}_{u,k}$ is such that
\be
|\hat{R}_{u,k}|\leq \frac{1}{2}B_{2d} \,\s{n}\max_{i\in [n]}|e_i| \times \sup_{(u,i)\in \mathcal{S}^{p-1}\times [n]}(\hat{U}^u_i - U^u_i)^2. \label{hat.Ruk.ubd}
\ee
Because $e_i$ and $X_i$ are independent, we have $E\{\bar{h}_{u,k}(\bar{X}_1,\bar{X}_2)|\bar{X}_1\}=E\{\bar{h}_{u,k}(\bar{X}_1,\bar{X}_2)|\bar{X}_2\}=0$ so that $\big\{ \sn \sum_{j=1}^n \bar{h}_{u,k}(\bar{X}_i, \bar{X}_j) \big\}_{(u,k)\in \mathcal{T}}$ forms a degenerate $U$-process. With slight abuse of notation, we rewrite the function $\bar{h}_{u,k}$ as $\bar{h}_{u,k}(\bar{x}_1, \bar{x}_2 ) = e_1 \cdot  w_{u,k}(x_1, x_2)$, where
$$
	  w_{u,k}(x_1, x_2) = \psi'_k(F^u(u^{\intercal}x_1)) \big[ I\{ u^{\intercal}(x_2-x_1) \leq 0 \} - F^u(u^{\intercal} x_1) \big].
$$
In this notation, we have $\bar{\mathcal{H}}^p_d:= \{	\bar{h}_{u,k}   : (u,k)\in \mathcal{T} \} \subseteq \{ e \mapsto e \} \cdot \mathcal{W}^p_d$ with $\mathcal{W}^p_d=\{ w_{u,k} : (u,k)\in \mathcal{T}\}$. Arguments similar to those employed in the proof of Lemma~\ref{VC.class.H} can be used to prove that the collection $\mathcal{W}^p_d$ is VC-type, and so is $\bar{\mathcal{H}}^p_d$ with envelop $\bar{H}$ given by $\bar{H}(\bar{x})=\bar{H}(e,x)=B_{2d}|e|$, such that
\be
	\sup_Q N\big(\bar{\mathcal{H}}^p_d, L_2(Q), \varepsilon \| \bar{H} \|_{Q,2} \big) \leq  d \cdot (A/\varepsilon)^{v p}  \label{entropy.bound.R}
\ee
for some constants $A>2e$ and $v \geq 2$. This uniform entropy bound, together with Theorem~6 in \cite{Nolan_Pollard_1987} yields
\begin{align}
	& E \bigg\{ \sup_{(u,k)\in \mathcal{T}} \bigg| \sn \sum_{j=1}^n \bar{h}_{u,k}(\bar{X}_i, \bar{X}_j) \bigg| \bigg\}   \nn \\
	& \lesssim  B_{2d} \, n  \bigg\{ \frac{1}{4} + \int_0^{1/4} \sup_Q \s{\log  N( \bar{\mathcal{H}}^p_d, L_2(Q), \varepsilon \| \bar{H} \|_{Q,2} ) } \, d \varepsilon \bigg\}  \lesssim B_{2d} \,n \s{p+\log d}  \label{expectation.degenerate.U.ubd}
\end{align}
by following the same lines as in the proof of Proposition~\ref{negligibility.degenrate.U.process}.

For $\hat{R}_{u,k}$, applying the Borell-TIS inequality gives
$$
	P \bigg\{  \max_{i\in [n]} |e_i| \leq  E\bigg(  \max_{i\in [n]} |e_i| \bigg) + t \bigg\} \leq \exp(-t^2/2)
$$
for every $t>0$. A standard result on Gaussian maxima is that $E(\max_{i\in [n]}|e_i|)\leq 2\s{\log n}$. Consequently, combining Proposition~\ref{negligibility.EDF} and \eqref{hat.Ruk.ubd} implies that
\be
	\sup_{(u,k)\in \mathcal{T}}  |\hat{R}_{u,k} | \lesssim  B_{2d} \, (\log n)^{1/2} \s{\frac{p+\log(dn)}{n} } \label{supremum.hatRuk.ubd}
\ee
holds with probability at least $1-3n^{-1}$,

By \eqref{marginal.multiplier.dec}, \eqref{expectation.degenerate.U.ubd} and \eqref{supremum.hatRuk.ubd}, a similar argument to that leading to \eqref{comparison.2} gives, on this occasion that for any $\gamma \in (0,1)$,
\begin{align}
P\bigg\{  \big| \hat{\Psi}_{\max}^{{\rm MB}} - {\Psi}_0^{\dagger} \big| \gtrsim  B_{2d} \big( \s{\log n} \vee \gamma^{-1} \big) \s{\frac{p+\log(dn)}{n}} \bigg\} \lesssim  \gamma +n^{-1}, \label{bootstrap.comparison.1}
\end{align}
where
\be
 \Psi^{\dagger}_0=\sup_{(u,k)\in \mathcal{S}^{p-1}\times [d]} \bigg| \frac{1}{\s{n}} \sn e_i \psi_k(U^u_i) \bigg| = \sup_{ f \in  {\mathcal{F}}} \bigg| \frac{1}{\s{n}} \sn e_i f(X_i) \bigg| \label{bootstrap.leading.term}
\ee
for ${\mathcal{F}}= {\mathcal{F}}^p_d$ as in \eqref{class.barF}.

Notice that $\Psi_0^{\dagger}$ is the supremum of a (conditional) Gaussian process $\bG^{\dagger}$ indexed by ${\mathcal{F}}$ with covariance function $E_e \{ (\bG^{\dagger} f_{u,k} )(\bG^{\dagger} f_{ v,\ell }  \} = n^{-1} \sn \psi_k(F^u(u^{\intercal} X_i)) \psi_\ell(F^u(u^{\intercal} X_i))$. Next we use an approximation due to \cite{chernozhukov2013anti}. Let $\mathcal{X}_n=\{X_1, \ldots, X_n\}$ be a realization of the data. Theorem~A.2 there shows that for every $\de>0$, there exists a subset $\Omega_n$ such that $P(\mathcal{X}_n \in \Omega_n) \geq 1-3 n^{-1}$ and for every $\mathcal{X}_n\in \Omega_n$, one can construct on an enriched probability space a random variable $\Psi^{\dagger}$ such that $\Psi^{\dagger} =_d  \| \bG \|_{ \mathcal{F}}$ for $\bG$ as in Lemma~\ref{Gaussian.approximation.prop} and that
\begin{align}
	 & P \bigg\{ \big| \Psi^{\dagger}_0 - \Psi^{\dagger} \big|  \gtrsim  \de+   \s{\frac{K^p_d \log n}{n}} +   B_{1d}^{1/2} \frac{(K^p_d \log n)^{3/4}}{n^{1/4}}   \bigg| \mathcal{X}_n \bigg\} \nn \\
	 & \qquad \qquad\qquad\qquad\qquad\qquad \qquad\qquad\qquad  \lesssim   B_{1d}^{1/2}\frac{(K^p_d \log n)^{3/4}}{\de n^{1/4}}  + n^{-1} , \label{bootstrap.comparison.2}
\end{align}
where $K^p_d=p+ \log d$.

Finally, combining \eqref{anti-concentration.bound} with inequalities \eqref{bootstrap.comparison.1} and \eqref{bootstrap.comparison.2}, and setting
$$
	\ga = B_{2d}^{1/2}  \{(p+\log n)/n\}^{1/4}  \ \ \mbox{ and }  \ \  \delta = B_{1d}^{1/4} (\log n)^{3/8} (p/n)^{1/8}
$$
complete the proof of \eqref{bootstrap.BE.bound} in view of \eqref{constraint.pd} and \eqref{anti-concentration.bound}. \qed.

\section{Proof of technical lemmas}
\label{aux.lemma}

We provide proofs here for all the technical lemmas. Throughout, we use $C$ and $c$ to denote universal positive constants, which may take different values at each occurrence.

\begin{lemma}  \label{lsw.2011}
{\rm
Let $\{\xi_i,i\geq 1\}$ be a sequence of independent random variables with zero means and finite variances. Put $S_n=\sn \xi_i$, $v_n^2=\sn \xi_i^2$ and $b_n^2 = \sn E (\xi_i^2)$, then for any $x > 0$,
\begin{align}
	P \big\{   |S_n|\geq x (v_n +4b_n ) \big\}   \leq  4    \exp( -x^2/2 ) \ \ \mbox{ and }    \label{lsw.1}  	\\
 E \big[ S_n^2 I\{ |S_n|\geq x (v_n+4b_n ) \} \big]    \leq 23   b_n^2 \, \exp( -x^2/4 ). \label{lsw.2}
\end{align}}
\end{lemma}

\begin{proof}[Proof of Lemma~\ref{lsw.2011}]
The proof is based on Theorem~2.16 in \cite{Pena_Lai_Shao_2009} and Lemma~3.2 in \cite{Lai_Shao_Wang_2011}.
\end{proof}

\begin{lemma} \label{lemma.1}
{\rm Assume that the conditions of Proposition~\ref{thm.limiting.dist} hold, then for every $k\geq 1$ and $t >0$,
\begin{align}
 P\big(  \s{nm} | R_{2k} | \geq  C_1 \| \psi'_k \|_\infty \, t \, \big) \leq C_2  \exp(- t/4 ), \label{R2k.tail.prob}
\end{align}
where $C_1, C_2 >0$ are absolute constants. }
\end{lemma}

\begin{proof}[Proof of Lemma~\ref{lemma.1}]

Without loss of generality we only prove the result for $t\geq 4$, otherwise we can simply adjust the constant $C_2$ so that $C_2 \exp(-t/4)\geq 1$ for $0\leq t\leq 4$. For given $k \geq 1$, define $Q_i=\sm q_{ij}$ with $q_{ij}= q_{ij,k}=h_{0k}(X_i, Y_j)$ for $h_{0k}$ as in \eqref{Hoeffding.dec}. Put  $\mathcal{F}_Y=\sigma\{Y_1,\ldots, Y_m\}$, such that given $\mathcal{F}_Y$, $\{ Q_i\}_{i=1}^n$ forms a sequence of independent random variables with zero (conditional) means. Noting that $\sn \sm q_{ij} =\sn Q_i$, it follows from a conditional version of \eqref{lsw.1} that for any $t\geq 4$,
\begin{align}
		P\bigg(  \bigg| \sn \sm q_{ij}  \bigg| \geq t \, \bigg[  \bigg( \sn Q_i^2 \bigg)^{1/2} + 4  \bigg\{ \sn E( Q_i^2 | \mathcal{F}_Y) \bigg\}^{1/2} \bigg]  \bigg| \mathcal{F}_Y \bigg) \leq 4\exp(-t^2/2).   \label{mtg.tail.prob}
\end{align}

We study the tail behaviors of $\sn Q_i^2$ and $\sn  E( Q_i^2 | \mathcal{F}_Y ) $ separately, starting with $\sn Q_i^2$. Observe that given $X_i$, $Q_i$ is a sum of independent random variables with zero means. Put $V_i^2= \sm q^2_{ij}$ and $B_i^2= \sm E( q^2_{ij}|X_i)$. A direct consequence of \eqref{lsw.2} is that for every $t>0$, $E [ Q_i^2 I\{ |Q_i| \geq t(V_i+4 B_i) \} | X_i  ] \leq  23 B_i^2   \exp(-t^2/4 )$. This implies by taking expectations on both sides that
\be
	E \big[ Q_i^2 I\{ |Q_i| \geq t(V_i+4 B_i) \} \big] \leq 23 E (B_i^2) \, \exp(-t^2/4), \label{exp.tail.ubd}
\ee
where $
	E (B_i^2) = \sm E ( q^2_{ij} ) \leq m E\{ h_k(X,Y)^2\} = m \sigma_k^2$ for $\sigma_k^2$ as in \eqref{kernel.property}. Together, \eqref{exp.tail.ubd} and Lemma~7.2 in \cite{Shao_Zhou_2014} imply, for $t\geq 4$,
\be
 P \bigg[ \sn Q_i^2 \geq t^2  \bigg\{  nm \sigma_k^2 +   \sn (V_i +4 B_i)^2 \bigg\}  \bigg] \leq 92 \,  t^{-4} \exp(-t^2/4) \leq (1/2)\exp(-t^2/4).   \label{sum.of.square.tail.prob}
\ee

We consider next $\sn  E( Q_i^2 | \mathcal{F}_Y )$, which can be decomposed as
\begin{align*}
	  & E( Q_i^2 | \mathcal{F}_Y )  \\
	  & =   E\big[ Q_i^2 I\{ | Q_i| \leq t(V_i+4 B_i) \} | \mathcal{F}_Y \big]  +   E\big[ Q_i^2 I\{|Q_i| > t(V_i+4 B_i) \} | \mathcal{F}_Y ) \big]  \\
	  & \leq t^2   E\big[  (V_i+4 B_i)^2 | \mathcal{F}_Y \big]  +   E\big[ Q_i^2 I\{|Q_i| > t(V_i+4 B_i) \} | \mathcal{F}_Y ) \big] \\
	  & \leq 17 t^2 m \sigma_k^2 + 17 t^2\sm E( q_{ij}^2| Y_j ) +   E\big[ Q_i^2 I\{|Q_i| > t(V_i+4 B_i) \} | \mathcal{F}_Y ) \big].
\end{align*}
Hence, it follows from Markov's inequality and \eqref{exp.tail.ubd} that
\begin{align}
	& P \bigg[  \sn E( Q_i^2 | \mathcal{F}_Y )  \geq 18 t^2 \bigg\{  nm \sigma_k^2 + \sn\sm E( q_{ij}^2| Y_j ) \bigg\}  \bigg] \nn \\
	& \leq  P\bigg( \sn E \big[ Q_i^2 I\{|Q_i| > t(V_i+4 B_i) \} | \mathcal{F}_Y ) \big] \geq t^2 nm\sigma_k^2 \bigg) \nn \\
	& \leq t^{-2}(nm \sigma_k^2)^{-1}\sn E\big[ Q_i^2 I\{ Q_i^2 > t^2(V_i+4B_i)^2 \} \big] \nn \\
	&  \leq  (3/2)\exp(-t^2/4). \label{sum.of.cond.square.tail.prob}
\end{align}

By \eqref{kernel.property}, we have $\| h_{0k} \|_\infty\leq 2 b_k$. Then combining \eqref{mtg.tail.prob}, \eqref{sum.of.square.tail.prob} and \eqref{sum.of.cond.square.tail.prob} gives, for $t\geq 4$,
\begin{align}
	 P\bigg\{ \bigg| \frac{1}{\s{nm}}\sn \sm q_{ij} \bigg| \geq  C_1(\sigma_k+ b_k)  t  \bigg\} \leq 6\exp(-t/4). \nn
\end{align}
This completes the proof of Lemma~\ref{lemma.1}.
\end{proof}

\begin{lemma} \label{alternative.prop}
Assume that the conditions of Theorem~\ref{asymptotic.power.1} are fulfilled, then for all sufficiently large $n$,
\be
	  \max_{k\in[d]}  | \vartheta_k - \theta_k  |  \lesssim  B_{0d}^2 \, d^{2\tau}  \frac{\log d}{n},
	   \label{alternative.mean.1}
\ee
where $\vartheta_k := E_{H^d_{1}}\{ \psi_k(V)  \}$ with $V=F(Y)$.
\end{lemma}

\begin{proof}[Proof of Lemma~\ref{alternative.prop}]
Under the alternative $H^d_{1}$, the density of $V=F(Y)$ is of the form $\rho_\theta(z)= C_d(\theta)\exp\{\theta^{\intercal} \bpsi(z)\}$, where $\{C_d(\theta)\}^{-1}=\int_0^1 \exp\{\theta^{\intercal} \bpsi(z)\} \, dz$ and $\bpsi=(\psi_1,\ldots, \psi_d)^{\intercal}$. In this notation, we have $\vartheta_k=C_d(\theta) \int_0^1 \psi_k(z) \exp\{\theta^{\intercal} \bpsi(z) \} \, dz $. Note that
$$
  |\theta^{\intercal} \bpsi(z)| =  \bigg| \sum_{k=1}^d \theta_{k} \psi_{k} (z)  \bigg|\leq   B_{0d}\, d^\tau \max_{k\in [d]} |\theta_k| = \lambda B_{0d} \, d^{\tau}\s{\frac{\log d}{\bar{n}}}.
$$
Consequently, using the inequality $|e^t-1-t|\leq  \frac{1}{2} t^2 \exp(t\vee 0)$ which holds for every $t\in \mathbb{R}$ to $t=  |\theta^{\intercal} \bpsi(z)| $ yields
\begin{align*}
 \int_0^1 \psi_k(z) \exp\{\theta^{\intercal} \bpsi(z)\} \, dz  &=  \int_0^1 \psi_k(z)  \{ 1+ \theta^{\intercal} \bpsi(z) \} \, dz + O(1) \int_0^1 |\psi_k(z)| \{ \theta^{\intercal} \bpsi(z)\}^2 \, dz    \\
 &= \sum_{\ell=1}^d \theta_\ell \int_0^1 \psi_k(z)\psi_\ell(z) \, dz + O(1) B_{0d}^2 \, d^{2\tau} \frac{\log d}{n} \\
	 &=  \theta_k + O(1) B_{0d}^2 \, d^{2\tau}  \frac{\log d}{n}
\end{align*}
uniformly over $k\in [d]$. Similarly, it can be proved that
\beq
	\bigg| \int_0^1 \exp\{\theta^{\intercal} \bpsi(z)\} \, dz -1 \bigg|  \lesssim B_{0d}^2 \, d^{2\tau} \frac{\log d}{n}  ,
\eeq
which implies $C_d(\theta)=1+o(1)$ as $d,n\rightarrow \infty$. Combining the above calculations proves~\eqref{alternative.mean.1}.
\end{proof}

\begin{lemma} \label{VC.class.H}
Under the null hypothesis $H_0: F=G$, the class $\mathcal{H}_0$ of degenerate kernels $\br^p \times \br^p \mapsto \br$, to which an envelop $\equiv 2B_{2d}$ is attached, is VC-type; that is, there are constants $A>2e$ and $v \geq 2$ such that
\be
	\sup_{Q \, {\rm discrete}}N \big(  \mathcal{H}_0, L_2(Q), 2\varepsilon B_{2d}  \big) \leq d \cdot ( A /  \varepsilon )^{v p}, 	\label{uniform.entropy.H}
\ee
where the supremum ranges over all finitely discrete Borel probability measures on $\br^p \times \br^p$.
\end{lemma}

\begin{proof}[Proof of Lemma~\ref{VC.class.H}]
First we prove that the class $\mathcal{H}$ of kernels is VC-type. Note that $\mathcal{H}$ has envelop $\equiv B_{2d}$ and admits the partition $\mathcal{H} = \cup_{k=1}^d \mathcal{H}_k$, where for each $k\in [d]$, the class $\mathcal{H}_k= \{ h_{u,k}  \in \mathcal{H} :  u\in \mathcal{S}^{p-1}  \}$ has an envelop $\leq B_{2d}$. This implies
\begin{align}
	  \sup_{Q \, {\rm discrete}} N \big(  \mathcal{H}, L_2(Q) ,  \varepsilon B_{2d}  \big)  \leq \sum_{k=1}^d \sup_{Q \, {\rm discrete}} N \big(  \mathcal{H}_k, L_2(Q) ,   \varepsilon B_{2d} \big) , \label{decomposition.bound}
\end{align}
where the supremum ranges over all finitely discrete Borel probability measures on $\mathbb{S}_2 := \br^p\times \br^p$. Therefore, it suffices to restrict attention to the class $\mathcal{H}_k$ with a fixed $k\in [d]$. For every $u\in \mathcal{S}^{p-1}$, observe that $h_{u,k}(x,y) = \psi'_k \circ f_u(y)  \cdot I_u(x,y)$. Regarding each element of $\psi'_k(\mathcal{F}^p) := \{ y \mapsto \psi'_k\circ f_u(y) :  f_u \in \mathcal{F}^p\}$ as a measurable function on $\mathbb{S}_2$, i.e. $  (x,y) \mapsto \psi'_k \circ f_u(y)$, we have $\mathcal{H}_k \subset \psi'_k(\mathcal{F}^p) \cdot  \mathcal{I}^p$ for $\mathcal{F}^p$ and $\mathcal{I}^p$ given in \eqref{class.FI}. Since both the classes $\mathcal{F}^p$ and $\mathcal{I}^p$ have envelop $\equiv 1$ and the function $\psi'_k$ is Lipschitz continuous, it follows from Lemma~A.6 and Corollary~A.1 in \cite{chernozhukov2012gaussian} that, for any $0<\varepsilon\leq 1$,
\be
 \sup_{Q \, {\rm discrete}} N \big( \psi'_k(\mathcal{F}^p) , L_2(Q) , \varepsilon  B_{2d}  \big) \leq \sup_{Q \, {\rm discrete}} N \big(  \mathcal{F}^p, L_2(Q) , \varepsilon \big) \label{preserve.1}
\ee
and
\begin{align}
& \sup_{Q \, {\rm discrete}} N\big(  \mathcal{H}_k , L_2(Q)  , 2\varepsilon  B_{2d} \big) \nn \\
&  \leq   \sup_{Q \, {\rm discrete}} N\big(   \psi'_k(\mathcal{F}^p), L_2(Q), \varepsilon B_{2d} \big)  \sup_{Q \, {\rm discrete}} N\big(  \mathcal{I}^p, L_2(Q), \varepsilon \big), \label{preserve.2}
\end{align}
where the suprema appeared above are taken over all finitely discrete Borel probability measures on $\mathbb{S}_2$.

In view of \eqref{preserve.1} and \eqref{preserve.2}, it remains to focus on the classes $\mathcal{F}^p$ and $\mathcal{I}^p$. Arguments similar to those in \cite{Sherman_1994} can be used to control the entropies of $\mathcal{I}^p$. To see this, define $\mathcal{V}=\{v(\cdot,\cdot, \cdot; u) : u\in \br^p\}$ and $\mathcal{W}=\{ w(\cdot,\cdot, \cdot; \gamma)  :  \gamma\in \br\}$, where $v(x,y,t;u )=u^{\intercal} x- u^{\intercal} y$ and $w(x,y,t; \gamma)=\gamma \, t$ for $x,y \in \br^p$ and $t\in \br$. Note that $\mathcal{V}$ (resp. $\mathcal{W}$) is a $p$-dimensional (resp. $1$-dimensional) vector space of real-valued functions on $\mathbb{S}_2 \times \br$. By Theorem~4.6 in \cite{Dudley_2014}, the class of sets of the form $\{z:v(z)>s\}$ or $\{z:v(z)\geq s\}$ with $v\in \mathcal{V}$ for some $s\in \br$ fixed is a VC class with index $p+1$. For every $u\in \mathcal{S}^{p-1}$,
\begin{align*}
{\rm graph}(I_u) & = \big\{ (x,y,t)\in  \mathbb{S}_2 \times \br : 0< t < I_u(x,y) \big\} \\
	& = \big\{ u^{\intercal} x-u^{\intercal} y \leq 0 \big\} \cap \{ t>0 \}  \cap \{t\geq 1\}^{{\rm c}} \\
	& = \{ v_1 > 0 \}^{{\rm c}} \cap \{ w_1 > 0 \} \cap \{w_2 \geq 1\}^{{\rm c}} ,
\end{align*}
where $v_1 \in \mathcal{V}$ and $w_1, w_2 \in \mathcal{W}$. Together with Lemma~9.7 in \cite{Kosorok_2008}, this implies that $\{{\rm graph}(I_u): I_u \in \mathcal{I}\}$ forms a VC class with index $\leq p+3$. Consequently, by Theorem~9.3 in \cite{Kosorok_2008}, there exist constants $a>2e$ and $c\geq 2$ such that
\be
	\sup_{Q} N \big( \mathcal{I}^p , L_2(Q), \varepsilon \big) \leq  ( a/\varepsilon  )^{c   p}  \label{unif.entropy.bound.1}
\ee
for any $0<\varepsilon\leq 1$, where the supremum is taken over all Borel probability measures on $\mathbb{S}_2$. For $\mathcal{F}^p$, applying Lemma~A.2 in \cite{Ghosal_Sen_van der Vaart_2000} combined with \eqref{unif.entropy.bound.1} gives
\be
	\sup_Q N\big( \mathcal{F}^p , L_2(Q), 2\varepsilon  \big) \leq \sup_Q N\big( \mathcal{I}^p, L_2(P_X \times Q) ,  \varepsilon^2 \big)  \leq  \big( \s{a}/\varepsilon \big)^{2c p}, \label{unif.entropy.bound.2}
\ee
where the supremum ranges over all Borel probability measures on $\br^p$.

Together, \eqref{decomposition.bound}--\eqref{unif.entropy.bound.2} imply the VC-type property of the class $\mathcal{H}$.

We consider next the class $\mathcal{H}_0$ of degenerate kernels under $H_0$, which admits a partition similar to \eqref{decomposition.bound}, i.e. $\mathcal{H}_0=\cup_{k=1}^d \mathcal{H}_{0 k}$. Observe that for each $(u,k)\in \mathcal{T}$,
\be
	\bar{h}_{u,k}(x,y) = h_{u,k}(x,y)  +   \psi_k \circ f^u(x)  + \phi_k \circ f^u(y)  ,   \ \    x, y \in \br^p.  \nn
\ee
where $h_{u,k} \in \mathcal{H}_{0 k} \subset \mathcal{H}$, $f^u \in \mathcal{F}$ and $\phi_k(s) : = - s \psi'_k(s)$ for $0\leq s\leq 1$. For any $u,v\in \mathcal{S}^{p-1}$ and $k\in [d]$, we have $| \phi_k \circ f_u(y) - \phi_k \circ f_v(y) |  \leq 2 B_{2d} | f_u(x) - f_v(y)|$. This, together with Lemma~A.6 in \cite{chernozhukov2012gaussian} yields
\be
	\sup_{Q \, {\rm discrete}} N \big( \phi_k(\mathcal{F}^p ) , L_2(Q), 2\varepsilon B_{2d} \big) \leq \sup_{Q \, {\rm discrete}} N \big(  \mathcal{F}^p  , L_2(Q), \varepsilon \big). \label{preserve.3}
\ee

On combing \eqref{preserve.1}, \eqref{preserve.2} and \eqref{preserve.3}, and recalling the permanence of the uniform entropy bound under summation that is implied by Lemma~A.6 in \cite{chernozhukov2012gaussian}, we obtain
\begin{align*}
 	& \sup_{Q \, {\rm discrete}} N \big(  \mathcal{H}_{0 k} ,L_2(Q), 3\varepsilon B_{2d}   \big) \nn \\
 	& \leq  \sup_{Q \, {\rm discrete}} N\big( \mathcal{H}_{0 k} , L_2(Q), \varepsilon B_{2d}  \big)  \\
 	& \qquad \times \sup_{Q \, {\rm discrete}} N \big( \psi_k(\mathcal{F}^p), L_2(Q), \varepsilon B_{2d} \big)   \sup_{Q \, {\rm discrete}} N\big( \phi_k(\mathcal{F}^p) , L_2(Q), \varepsilon  B_{2d} \big) \\
 	& \leq    \sup_{Q \, {\rm discrete}} N \big(\mathcal{I}^p, L_2(Q), \varepsilon / 2  \big)   \bigg\{ \sup_{Q \, {\rm discrete}}  N\big(  \mathcal{F}^p , L_2(Q), \varepsilon/ 2 \big) \bigg\}^3.
\end{align*}
This completes the proof of \eqref{uniform.entropy.H} in view of \eqref{unif.entropy.bound.1} and \eqref{unif.entropy.bound.2}.
\end{proof}

\subsection{Proof of Lemma~\ref{negligibility.degenrate.U.process}}

Observe that $\{U_{m,n}( h_0 )\}_{ h_0 \in \mathcal{H}_0}$ forms a degenerate two-sample $U$-process indexed by $\mathcal{H}_0$ and by Lemma~\ref{VC.class.H}, $\mathcal{H}_0$ is VC-type with envelop $\equiv 2 B_{2d}$. The entropy bound given in \eqref{uniform.entropy.H} now allows us to apply Lemma~2.4 in \cite{Neumeyer_2004}, yielding
\begin{align}
 & 	E \big(  \| U_{n,m}  \|_{\mathcal{H}_0} \big) \nn \\
 &  \lesssim  B_{2d} \s{nm} \bigg\{ \frac{1}{4}  +   \int_0^{1/4} \sup_{Q \, {\rm discrete}}\s{\log N (\mathcal{H}_0, L_2(Q), 2\varepsilon B_{2d})} \, d\varepsilon \bigg\} \nn \\
 &\lesssim B_{2d} \s{nm} \bigg\{ \frac{1}{4}+  \int_0^{1/4}  \s{ \log d  + v p \log(A/\varepsilon) } \, d\varepsilon \bigg\} \nn \\
 & \lesssim  B_{2d} \s{nm}\bigg\{ \frac{1}{4}  \big(1+\s{\log d} \big)+  \s{vp} \int_{4A }^\infty t^{-2}\s{\log t } \, dt \bigg\} . \label{mean.Unm.ubd}
\end{align}
For any $a>e$, it follows from integration by parts that
\begin{align}
 & \int_a^\infty t^{-2}\s{\log t} \, dt   = a^{-1}\s{\log a} + \frac{1}{2} \int_a^\infty t^{-2}(\log t)^{-1/2} \, dt \nn \\
& \leq  a^{-1} \s{\log a} + \frac{1}{2\log a} \int_a^\infty t^{-2} \s{\log t} \, dt \leq   a^{-1} \s{\log a} + \frac{1}{2} \int_a^\infty t^{-2} \s{\log t} \, dt. \nn
\end{align}
Substituting this into \eqref{mean.Unm.ubd} proves \eqref{mean.sup.Unm.ubd}. \qed

\subsection{Proof of Lemma~\ref{negligibility.EDF}}

Define the class $\mathcal{G} = \{ x \mapsto g_{u,t}(x) = I(u^{\intercal} x \leq t ) : (u,t)\in \mathcal{S}^{p-1} \times \br  \}$ of indicator functions on closed half-spaces in $\br^d$, such that
\beq
	D_{n}(\mathcal{G}) := \sup_{(u,t)\in \mathcal{S}^{p-1}\times \br}  | F^u_n(t) - F^u(t)   | =  \sup_{g \in \mathcal{G}}  \bigg| \frac{1}{n} \sn \{ g(X_i) - P_X g  \} \bigg|,
\eeq
where $P_X g := E \{g(X)\}$. Note that, for every $(u,t)\in \mathcal{T}$, $\var\{g_{u,t}(X)\}=F^u(t)\{1-F^u(t)\}\leq 1/4$. A direct consequence of Theorem~7.3 in \cite{Bousquet_2003} is that, for every $t\geq 0$,
\begin{align}
	P\bigg(  D_{n}(\mathcal{G}) \geq E  \{ D_{n}(\mathcal{G}) \} +\big[ \tfrac{1}{2} + 4E \{ D_{n}(\mathcal{G}) \}  \big]^{1/2} \s{\frac{t}{n}}+  \frac{t}{3n}  \bigg)  \leq 2 e^{-t}.  \nn
\end{align}

To control the expectation $E \{ D_{n}(\mathcal{G}) \} $, first it follows from Theorem~B in \cite{dudley1979balls} that the class $\mathcal{G}$ is a VC-subgraph class with index $p+2$, such that for any probability measure $Q$ on $\br^p$ and any $0<\varepsilon<1$, $N ( \mathcal{G}, L_2(Q),  \varepsilon ) \leq (C \varepsilon^{-1})^{2(p+1)}$. This, together with Proposition~3 in \cite{Gine_Nickl_2009} gives
\begin{align}
	E  \big\{ D_{n}(\mathcal{G})  \big\} \lesssim    \s{\frac{p}{n}} + \frac{p}{n}  .   \nn
\end{align}
Since $D_n(\mathcal{G})\leq 1$, the last three displays together complete the proof of \eqref{edf.uniform.rate}. \qed

\subsection{Proof of Lemma~\ref{Gaussian.approximation.prop}}

To prove \eqref{CCK.Gaussian.approximation}, a new coupling inequality for the suprema of empirical processes in \cite{chernozhukov2012gaussian} plays an important role in our analysis. Recall in the proof of Lemma~\ref{VC.class.H} that the collection $\mathcal{F}^p$ is VC-type, from which we obtain
\be
	 \sup_{Q \, {\rm discrete}} N\big(  {\mathcal{F}}  , L_2(Q), \varepsilon  B_{1d} \big) \leq  \sum_{k=1}^d \sup_{Q \, {\rm discrete}} N\big( \psi_k(\mathcal{F}^p), L_2(Q), \varepsilon B_{1d} \big) \leq d \cdot (A / \varepsilon)^{v p}  \label{barF.entropy.bound}
\ee
for some constants $A>2e$ and $v\geq 2$, where $\mathcal{F}= \mathcal{F}^p_d $. This implies by Lemma~2.1 in \cite{chernozhukov2012gaussian} that the collection ${\mathcal{F}}$ is a VC-type pre-Gaussian class with a constant envelop $\equiv B_{1d}$. Therefore, there exists a centered, tight Gaussian process $\mathbb{G}$ defined on $ {\mathcal{F}}$ with covariance function \eqref{covariance.function}. Moreover, for any integer $k\geq 2$,
\begin{align}
	\sup_{ f \in  {\mathcal{F}}}  P_X |f|^k  \leq B_{0d}^{k-2}\sup_{ (u,k)\in \mathcal{S}^{p-1}\times [d] } E   \big\{ \psi_k(U^u)^2 \big\} = B_{0d}^{k-2}, \label{unif.moment.bd}
\end{align}
where we used the fact that $U^u=F^u(u^{\intercal} X) =_d {\rm Unif}(0,1)$ and hence $E   \{ \psi_k(U^u)^2 \}=1 $ for all $(u,k)\in \mathcal{S}^{p-1}\times [d]$.

The entropy bound \eqref{barF.entropy.bound} and the moment inequality \eqref{unif.moment.bd} now allow to apply Corollary~2.2 in \cite{chernozhukov2012gaussian}, yielding \eqref{CCK.Gaussian.approximation}. \qed

\end{document}